\documentclass[article,11pt]{amsart}

\usepackage{enumitem}
\usepackage{amsthm}
\usepackage{amssymb}
\usepackage{amstext}
\usepackage{amsmath}
\usepackage{url}
\usepackage[colorlinks]{hyperref}
\hypersetup{
linkcolor=blue,
citecolor=blue,
}
\usepackage[margin = 40pt,font=small,labelfont=bf]{caption}
\usepackage{arydshln}
\usepackage{pdfsync}
\usepackage[T1]{fontenc}
\usepackage{bbm}
\usepackage{ae,aecompl}
\usepackage{pdfpages}
\usepackage{epstopdf}
\usepackage{graphicx}
\usepackage{epsfig}
\usepackage{subfigure}
\usepackage{stmaryrd}
\usepackage{cite}
\usepackage{euscript}
\usepackage[dvipsnames]{xcolor}
\usepackage{todonotes}
\usepackage{ulem}
\usepackage{dsfont}
\usepackage[ruled,vlined]{algorithm2e}

\numberwithin{equation}{section}
\theoremstyle{plain}
\newtheorem{thm}{Theorem}[section]
\newtheorem{rem}{Remark}[section]
\newtheorem{prop}{Proposition}[section]
\newtheorem{cor}{Corollary}[section]
\newtheorem{lem}{Lemma}[section]
\newtheorem{definition}{Definition}[section]

\newtheorem{exa}{Example}[section]

\def\build#1_#2^#3{\mathrel{\mathop{\kern 0pt#1}\limits_{#2}^{#3}}}

\def\videbox{\mathbin{\vbox{\hrule\hbox{\vrule height1.4ex \kern.6em\vrule height1.4ex}\hrule}}}

\newcommand{\ST}{\mathrm{ST}(G)}
\newcommand{\Lip}{\mathrm{Lip}_1(d_G)}
\newcommand{\LipT}{\mathrm{Lip}_1(d_T)}
\newcommand{\LipTast}{\mathrm{Lip}_1(d_{T_\ast})}
\newcommand{\ArEl}{\mathrm{AE}}
\newcommand{\child}{\mathrm{child}}

\def\supp{\mathop{\rm supp}\nolimits}

\newcommand{\R}{{\mathbb R}}
\newcommand{\RR}{{\mathbb R}}
\newcommand{\N}{{\mathbb N}}

\newcommand{\FF}{\ensuremath{\mathcal F}}

\newcommand{\ee}{\ensuremath{\mathcal E}}

\newcommand{\PP}{\ensuremath{\mathcal P}}
\newcommand{\GG}{\ensuremath{\mathcal G}}

\newcommand{\XX}{\ensuremath{\mathcal X}}
\newcommand{\YY}{\ensuremath{\mathcal Y}}
\newcommand{\ZZ}{\ensuremath{\mathcal Z}}

\def\argmax{\mathop{\rm arg \; max}\limits}%
\def\sgn{\mathop{\rm Sgn}\nolimits}%

\newcommand{\thefont}[2]{\fontsize{#1}{#2}\fontshape{n}\selectfont}
\newcommand{\1}{\rlap{\thefont{10pt}{12pt}1}\kern.16em\rlap{\thefont{11pt}{13.2pt}1}\kern.4em}

\begin{document}
\title[]{Kantorovich Distance via Spanning Trees: Properties and Algorithms  \vspace{1ex}}
\author{J\'{e}r\'{e}mie Bigot and Luis Fredes}
\dedicatory{\normalsize Universit\'e de Bordeaux \\
Institut de Math\'ematiques de Bordeaux et CNRS  (UMR 5251)}
\thanks{The authors gratefully acknowledge financial support from the Agence Nationale de la Recherche  (MaSDOL grant ANR-19-CE23-0017).}

\maketitle

\thispagestyle{empty}

\begin{abstract}
We study optimal transport between probability measures supported on the same finite metric space, where the ground cost is a distance induced by a weighted connected graph. Building on recent work showing that the resulting Kantorovich distance can be expressed as a minimization problem over the set of spanning trees of this underlying graph, we investigate the implications of this reformulation on the construction of an optimal transport plan and a dual potential based on the solution of such an optimization problem. In this setting, we derive an explicit formula for the Kantorovich potential in terms of the imbalanced cumulative mass (a generalization of the cumulative distribution in $\R$) along an  optimal spanning tree solving such a minimization problem, under a weak non-degeneracy condition on the pair of measures that guarantees the uniqueness of a dual potential.

Our second contribution establishes the existence of an optimal transport plan that can be computed efficiently by a dynamic programming procedure
once an optimal spanning tree is known.
Finally, we propose a stochastic algorithm based on simulated annealing on the space of spanning trees to compute such an optimal spanning tree.
Numerical experiments illustrate the theoretical results and demonstrate the practical relevance of the proposed approach for optimal transport on finite metric spaces.
\end{abstract}

 \\


\section{Introduction}

This paper focuses on optimal transport (OT) between probability measures supported on the same finite metric space, where the ground distance is induced by a weighted connected graph~$G$. 
Recent work~\cite{pistone22} has shown that, using the fact that this OT problem admits a closed-form solution when~$G$ is a tree~\cite{EM12,MV23,Tam19,Sato20,Chen24}, the computation of the Kantorovich distance (i.e., the 1-Wasserstein distance induced by this ground metric) on an arbitrary graph $G$ can be reformulated as a minimization problem over the set of all spanning trees of $G$.
The present work is motivated by the research question introduced in~\cite{pistone22} regarding the potential computational applications of this result.
To this end, we first extend the analysis in~\cite{pistone22} by studying the properties of the optimal spanning trees of~$G$ which solve such a minimization problem, and thus OT on the probability simplex associated with the finite metric space.
We then propose a stochastic algorithm based on simulated annealing over the space of spanning trees to estimate the Kantorovich distance (K-distance), an associated optimal transport plan and a dual potential. 

\subsection{Kantorovich distance on a finite metric space}

Before presenting the main contributions of this work, we first provide some background on the K-distance  on a finite metric space. Let $\XX = \{X_1,\ldots,X_N\}$ denote a set of  $N$  points (that are pairwise distinct) endowed with a distance $d_G$ that is induced by an edge-weighted connected graph $G = (\XX,\ee_G,w)$  whose vertices (or nodes) are the elements of $\XX$, and $\ee_G$ denotes the set of edges
where $w : \ee_G \to \RR^+$ is a given weight function that is symmetric and positive that is $w(x,y) > 0$ for all $(x,y) \in  \ee_G$, with $x \neq y$, and $w(x,x) = 0$ for all $x \in \XX$.
This means that $\XX$ is a finite metric space endowed with the distance $d_G(x,y)$ that is equal to the length of the shortest path, according to $w$, connecting $x$ and $y$; in particular, $d_G(x,x)=0$ for every $x\in \XX$. We recall that a path  between two vertices $x$ and $y$ is a sequence of at least two vertices $x_0 = x,x_1,\ldots,x_m=y$ such that $\{x_{i-1},x_i\} \in \ee_G$ for all $i = 1,\ldots,m$, and its length is $\sum_{i=1}^{m} w(x_{i-1},x_i)$. Throughout the paper, we assume that the graph $G$ is connected (implying the existence of a path between any pair of vertices) and loops have distance zero, meaning that $w(x,x) = 0$ for all $x \in \XX$. We also consider $\Vec{\ee}_G$ the set of oriented edges of $G$, meaning that for every edge $\{x,y\}\in \ee_G$ there are two edges $(x,y)$ and $(y,x)$ in $\Vec{\ee}_G$.

We denote by $\PP(\XX)$ the set of probability measures supported on $\XX$. For $\mu \in \PP(\XX)$, we write $\mu = (\mu(x): x \in \XX)$ and we interpret it as a vector in the $N$-dimensional simplex, meaning that $\mu(x) \geq 0$ and $\sum_{x \in \XX} \mu(x) = 1$. For a subset $\XX' \subset \XX$ of vertices, we also write $\mu(\XX' ) = \sum_{x \in \XX'} \mu(x)$. Given a pair $(\mu,\nu)$ of such probability measures, the space $\Pi(\mu,\nu)$ of admissible transport plans (also called transportation polytope) is defined as the set    of joint probability measures $\gamma$ supported on $\XX \times \XX$ having marginals equal to $\mu$ and $\nu$, that is satisfying the constraints
$$
\sum_{y \in \XX} \gamma(x,y) = \mu(x), \quad \sum_{x \in \XX} \gamma(x,y) = \nu(y).
$$

\begin{definition}
For the distance $d_G$ induced by the graph $G$, the Kantorovich distance between two probability measures $\mu$ and $\nu$ in $\PP(\XX)$ is the value of the Linear Programming (LP) problem
\begin{equation}
K_{d_G}(\mu,\nu) = \min_{\gamma \in \Pi(\mu,\nu)} \sum_{x, y \in \XX} d_G(x,y) \gamma(x,y). \label{eq:OT_Kant}
\end{equation}
\end{definition}

The LP program \eqref{eq:OT_Kant} is known as the optimal transport (OT) problem between $\mu$ and $\nu$ in the sense of Kantorovich. We refer to \cite{Santambrogio2015,Villani} for a mathematical introduction to OT, and to \cite{peyre2020computational} for a detailed presentation of its computational aspects. In the OT literature, the distance $K_{d_G}$ is also referred to as the Wasserstein distance with respect to the ground cost $d_G$. A transport plan (also called coupling) $\gamma_G$ is called optimal if it is a minimizer of  the LP program \eqref{eq:OT_Kant}. A minimizer of this problem is not necessarily unique.

We also recall the well-known property of Kantorovich duality in the OT literature.
\begin{prop}
 The K-distance \eqref{eq:OT_Kant} also satisfies
\begin{equation}
K_{d_G}(\mu,\nu) = \max_{u \in \Lip} \sum_{x \in \XX} u(x) (\mu(x) - \nu(x)), \label{eq:dual_Kant} 
\end{equation}
where $ \Lip$ denotes the set of functions $u$ on $\XX$ satisfying $|u(x) - u(y)| \leq d_G(x,y)$ for all $(x,y) \in \XX \times \XX$.
\end{prop}
A function $u_G \in \Lip$   is called an optimal dual (or Kantorovich) potential if it is a maximizer of  the optimisation problem \eqref{eq:dual_Kant}. It is well known in the OT literature that simple counterexamples exist where the uniqueness of Kantorovich potentials fails to hold for standard ground costs. Nevertheless, conditions ensuring uniqueness (up to an additive constant) of dual potentials for OT on finite spaces have long been established using techniques from linear programming \cite{KW68,Hung1986}.  The uniqueness of Kantorovich potentials for arbitrary probability measures (beyond the discrete setting) has also been addressed in the recent work \cite{Staudt25} to which we refer for a detailed overview of existing works on this topic. Following these works, the uniqueness of Kantorovich potentials typically arises whenever the measures (not necessarily supported on the same space) are non-degenerate in the following sense. 

\begin{definition}\label{def:non-deg}
Let $\XX$ and $\YY$ be two discrete sets of points. A pair of probability measures $(\mu,\nu) \in  \PP(\XX) \times \PP(\YY)$ is said to be non-degenerate if for any non-empty proper subsets $\XX'$ of $\XX$ and $\YY'$ of $\YY$, it holds that  $\mu(\XX')\neq \nu(\YY')$.
\end{definition}

\subsection{Main contributions}

In this paper, we build on the known connection between the K-distance and the Arens-Eells (AE) norm~\cite{Weaver18}, together with a recent result by~\cite{pistone22} showing that solving the linear program (LP)~\eqref{eq:OT_Kant} is equivalent to the following minimization problem over the set~$\ST$ of rooted spanning trees of the graph~$G$:
\begin{equation} \label{min:ST}
K_{d_G}(\mu,\nu) = \min_{T \in \ST} K_{d_T}(\mu,\nu),
\end{equation}
where $K_{d_T}(\mu,\nu)$ denotes the K-distance between $\mu$ and $\nu$, which admits a closed-form expression when the ground cost is given by the distance~$d_T$ induced by a rooted spanning tree~$T \in \ST$ with edge weights defined by the function~$w$. 
Extending the analysis developed in~\cite{pistone22}, our main result (see Theorem \ref{theo:unique-potential}) establishes a connection between an optimal spanning tree~$T_\ast$ minimizing~\eqref{min:ST} and an optimal Kantorovich potential, under the following weak non-degeneracy condition on the measures~$\mu$ and~$\nu$.

\begin{definition}\label{def:weak-non-deg}
A pair of probability measures $\mu, \nu \in \PP(\XX)$ is said to be weakly non-degenerate if, for every non-empty proper subset $\XX' \subset \XX$, it holds that $\mu(\XX') \neq \nu(\XX')$.
\end{definition}

\subsubsection{An explicit expression of the Kantorovich potential}

In particular, we show that when the pair of  measures is weakly non-degenerate,  the Kantorovich potential $u_G$ is unique up to an additive constant. 
Definition~\ref{def:weak-non-deg} of weak non-degeneracy exponentially reduces the number of conditions to be verified compared with Definition~\ref{def:non-deg} of non-degeneracy, when $\XX=\YY$. 
Moreover, under the weak non-degeneracy of the pair $(\mu,\nu)$, the Kantorovich potential $u_G$ admits an explicit form that depends on the optimal spanning tree $T_\ast$  rooted at $r \in \XX$. More precisely, we show that (under weak non-degeneracy) for every $y \in \XX$, 
 \begin{equation}\label{eq:pot_intro}
 u_G(y) = \sum_{y \lesssim x \neq r} d_{T_\ast}(x,x^{+}) \sgn ( \Xi_{T_\ast}(x) ) \text{ with the convention that  } u_G(r) = 0,
\end{equation}
where $y \lesssim x$ denotes the  partial order  induced by the rooted tree $T_\ast$ (with edges oriented towards the root), $x^{+}$ is the unique parent of $x \neq r$ (see Section \ref{sec:tree}, for more precise definitions), $\sgn(\cdot)$ is the usual sign function, and the quantity $\Xi_{T_\ast}(x)$ is the imbalanced cumulative function defined as
 \begin{equation}
 \Xi_{T_\ast}(x) = \sum_{y \lesssim x} \mu(y) - \nu(y), \quad \mbox{ for all } x \in \XX. \label{eq:main_xi}
\end{equation}
The function $ \Xi_{T_\ast}$ is the difference between cumulative functions  (with respect to the partial order induced by $\lesssim$ on  $T_\ast$)
$$
F_{\mu}(x) = \sum_{y \lesssim x} \mu(y) \text{ and } F_{\nu}(x) = \sum_{y \lesssim x} \nu(y)
$$
of the measures $\mu$ and $\nu$. This generalizes the standard cumulative functions of probability measures supported on the real line.

\subsubsection{The analogy with monotone optimal transport on the real line}

The expression \eqref{eq:pot_intro} of the Kantorovich potential $u_G$  also reveals a structural analogy between OT with respect to a graph-induced ground metric and monotone transport on the real line. To illustrate this connection, consider the special case where   $\XX$ is a set of reals $X_1 < X_2 < \ldots < X_N$ and the ground metric is $d_G(x,y) = |x-y|$. This corresponds to the setting where $G$ is a line tree. As $G$ is already a tree, the optimal spanning tree $T_\ast$ of problem \eqref{min:ST} is the line tree $G$ itself. We will see in Section \ref{sec:tree} that the computation of the $K$-distance is invariant under rerooting so we can impose without loss of generality the root $r= X_N$, then
$$
X_1 \to X_2 \to \cdots \to X_N \quad \text{with} \quad X_{i+1} = X_i^+.
$$
In this one-dimensional setting and for this ground cost, it is well known (see e.g.\ \cite{Santambrogio2015}[Chapter 2]) that the $K$-distance between one-dimensional measures $\mu$ and $\nu$ supported on $\RR$    is
$$
K_{d_G}(\mu,\nu) = \int_\RR |F_{\mu}(x) - F_{\nu}(x)| dx \quad \text{ with } F_{\mu}(x) = \mu((-\infty;x])  \text{ and }  F_{\nu}(x) = \nu((-\infty;x]).
$$
Therefore,
$$
K_{d_G}(\mu,\nu)  =  \sum_{i=1}^{N-1} (X_{i+1}-X_i) |F_{\mu}(X_i) - F_{\nu}(X_i)| 
 =  \sum_{x \in \XX \setminus \{r\}} d_{G}(x,x^+) \sgn ( \Xi_{T_\ast}(x) ) \Xi_{T_\ast}(x),
$$
given that in the case of the line tree, the imbalance function \(\Xi_{T_\ast}\) defined in \eqref{eq:main_xi} satisfies
$\Xi_{T_\ast}(X_i) = F_{\mu}(X_i)-F_{\nu}(X_i)$. Since, for every $x \in \XX$,
$
\Xi_{T_\ast}(x) =  \sum_{y \lesssim r} \left( \mu(y) -\nu(y) \right) \mathbf{1}_{\{y \lesssim x\}}
$
and $d_{G}(x,x^+) = d_{T_\ast}(x,x^+) $, one has that
\begin{eqnarray*}
K_{d_G}(\mu,\nu)  & = & \sum_{y \lesssim r}\sum_{x \in \XX \setminus \{r\}} d_{T_\ast}(x,x^+) \sgn ( \Xi_{T_\ast}(x) )  \mathbf{1}_{\{y \lesssim x\}} \left( \mu(y) -\nu(y) \right) \\
 & = &\sum_{y \lesssim r}  u_G(y)  \left( \mu(y) -\nu(y) \right),
\end{eqnarray*}
with $u_G(y) =  \sum_{y \lesssim x \neq r} d_{T_\ast}(x,x^+) \sgn ( \Xi_{T_\ast}(x) )$ which is the Kantorovich potential in the one-dimensional case satisfying $u_G(X_N) = 0$. This highlights the fact that,  given the knowledge of an optimal spanning tree $T_\ast \in \ST$ minimizing \eqref{min:ST}, the expression \eqref{eq:pot_intro} of a Kantorovich potential generalizes the structure of one-dimensional monotone transport.

\subsubsection{The construction of an optimal coupling on a weighted tree by dynamic programming}

In addition,  a second contribution is to derive a new algorithm to compute an optimal transport plan $\gamma_G$ in the case where $G = T$ is a tree. Our approach differs from the standard use of the network simplex algorithm implemented in the Python Optimal Transport (POT) library ~\cite{flamary2024pot,flamary2021pot} based on the cost matrix between the elements of $\XX$. Such an OT plan can be easily computed through an algorithmic procedure based on dynamic programming and the knowledge of the imbalance cumulative function $\Xi_{T}(\cdot)$.

\subsubsection{Computation of an optimal spanning tree of \eqref{min:ST} by simulated annealing}

Finally, we propose a stochastic algorithm to compute an optimal spanning tree $T_\ast$. The algorithm relies on a random walk, based on simulated annealing, over the set~$\ST$ of spanning trees in order to optimize the functional $T \mapsto K_{d_T}(\mu,\nu)$. This immediately yields an estimation  of the $K$-distance. Then, once $T_\ast$ has been found by simulated annealing, we can compute an OT plan thanks to our dynamic programming based on  the knowledge  of  $\Xi_{T_\ast}(\cdot)$. Finally, the computation of a Kantorovich potential follows from the explicit expression \eqref{eq:pot_intro}. 

\subsection{Organization of the paper}

In Section \ref{sec:Arens}, we present the connection between the K-distance and the Arens-Eells norms as highlighted in  \cite{pistone22}.  Section \ref{sec:contributions} contains the theoretical contributions of the paper on the properties of an optimal spanning tree $T_{\ast}$ minimizing \eqref{min:ST}. In Section \ref{Algos} we describe two algorithms; one based on dynamic programming to compute an optimal transport plan when $G = T$ is a tree, and one to find the structure of the optimal spanning tree $T_{\ast}$ minimizing \eqref{min:ST} based on simulated annealing. Numerical experiments to illustrate the theoretical results of the paper are reported in Section \ref{sec:num}. A conclusion and some perspectives are given in Section \ref{sec:conclusion}.

\section{A connection between the Kantorovich distance and the Arens-Eells norms} \label{sec:Arens}

In this section, we first recall useful properties of optimal transport plans. Then, we recall the notion of Kantorovich duality. Finally, we detail the existing links between the K-distance on a finite metric space and  the notion of Arens-Eells norms using notation that are mainly borrowed from \cite{pistone22}. 

\subsection{Properties of optimal transport plans}

For a given transport plan $\gamma$, its oriented support will be denoted as
$$
\overrightarrow{\supp}(\gamma) := \left\{ (x,y) \mid \gamma(x,y) > 0 \;: \; (x,y) \in \XX \times \XX \right\}.
$$

Throughout the paper, we also refer to $\overrightarrow{\supp}(\gamma)$ as the directed graph with vertices $\XX$ such that the oriented edge $(x,y)$ is in the support of $\gamma$ if and only if $\gamma(x,y) > 0$. It is important to remark that the edge $(x,y)$ belonging to the support of $\gamma$ lives in $\XX\times \XX$, meaning that $\overrightarrow{\supp}(\gamma)$ is a subgraph of the complete graph $K_{\XX}$ together with an orientation of its edges. We also denote by $\supp(\gamma) $ the corresponding non-oriented graph associated to a transport plan $\gamma$ consisting of  non-oriented edges $\{x,y\}$ such that $\gamma(x,y) > 0$. Denoting $\mathcal{S}_{m}$ the permutation group of size $m \in\N$,  let us now recall the definition of $d_G$-cyclical monotonicity. 

\begin{definition}
A subset $\Gamma \subset \XX\times \XX$ is said to be $d_G$-cyclically monotone if for any $m\in\N$, every permutation $\sigma \in \mathcal{S}_{m}$  and  every finite family of points $(x_0,y_0),(x_1,y_1),\dots, (x_{m-1},y_{m-1})$ in $\Gamma$ the following inequality holds
\[
    \sum_{i=0}^{m-1} d_G(x_i,y_i) \leq \sum_{i=0}^{m-1} d_G(x_i,y_{\sigma(i)}).
\]
\end{definition}

Then, it is well known (see e.g.\ Theorem 5.9 in \cite{villani2}) that the oriented support $\overrightarrow{\supp}(\gamma_G)$ of an optimal transport plan $\gamma_G$ is necessarily $d_G$-cyclically monotone. Therefore, the support of an optimal plan $\overrightarrow{\supp}(\gamma_G)$ is necessarily a union of directed acyclic graphs (when forgetting loops).  Indeed, checking that there is no oriented cycle comes easily from the $d$-cyclically monotone property.  If there were  an oriented cycle $x_0,x_1,\ldots,x_{m}$ in $\supp(\gamma_G)$ of size $m+1$ with $x_{m} = x_0$ that is not a loop (i.e.\ such that $d_G(x_i,x_{i+1}) > 0$ for at least one $i$), then by the $d_G$-cyclical monotonicity of  $\supp(\gamma_G)$ and the permutation $\sigma(i) = i-1$ for  $i \in \{1,\ldots,m-1\}$ and $\sigma(0) = m-1$, we would obtain the inequalities
$$
0 < \sum_{i=0}^{m-1} d_G(x_i,x_{i+1}) \leq \sum_{i=0}^{m-1} d_G(x_i,x_i) = 0,
$$
which is a contradiction. However, the non-oriented graph $\supp(\gamma_G)$ is not necessarily a forest, that is a union of trees. Nevertheless, this is more intricate as we shall now prove that one may always construct an OT plan such that  $\supp(\gamma_G)$ is a forest. Throughout the paper, we consider trees as connected and non-oriented graphs without cycles but with loops being allowed, since keeping mass at a node has a zero cost and this will turn out to be  relevant  for the construction  of such an OT plan when $G=T$ is a tree as highlighted in Proposition \ref{prop:coldiag} and Section \ref{sec:OTplan}.%

To this end, let us recall that, to any transport plan $\gamma$ between two probability measures $\mu\in \PP(\YY)$ and $\nu\in\PP(\ZZ)$ not necessarily supported on the same finite spaces $\YY$ and $\ZZ$, one may associate a bipartite graph denoted $\GG_{\gamma}(\YY,\ZZ) = (\YY \cup \ZZ,E_\gamma)$ in what follows, where the vertex set consists of one copy of $\YY$ and one copy of $\ZZ$ and the non-oriented edge $\{y,z\} \in E_\gamma$ if $\gamma(y,z)>0$. The following result explains the structure of this graph when the transport plan is a vertex of the transportation polytope.
\begin{thm}[Theorem 4 in \cite{KW68}] \label{thm:forest}
    Let  $\gamma$ denote any vertex of the transportation polytope $\Pi(\mu,\nu)$ between two probability measures $\mu\in \PP(\YY)$ and $\nu\in\PP(\ZZ)$. Then $\GG_\gamma(\YY,\ZZ)$ is a bipartite forest.
\end{thm}

Returning to the OT problem considered in this paper with $\YY = \ZZ = \XX$, this lets us deduce that an OT plan $\gamma_G$, that is a vertex of the  polytope $\Pi(\mu,\nu)$, induces a graph $\GG_{\gamma}(\XX,\XX')$ that is a bipartite forest on a set of vertices consisting of two copies of $\XX$. We will show that there always exists an optimal solution such that $\supp(\gamma_G)$ is a forest (together with self-loops) in $K_{\XX}$, i.e. in the graph where there is only one copy of $\XX$. To this end, we use the following result showing that every solution of the transportation problem can be turned into a solution which maximizes the diagonal entries and that moreover is a vertex of the transportation polytope $\Pi(\mu,\nu)$. Intuitively, one copy of $\XX$ in $\GG_{\gamma}(\XX,\XX')$ gives the behavior of each vertex when it sends goods and the second copy $\XX$ gives the behavior of each vertex as receiver of goods, and since in this problem there is zero cost of keeping goods on site, we will maximize the goods kept at site to obtain a diagonal maximizing solution.

\begin{prop}\label{prop:coldiag}
Consider the OT problem with ground cost $d_G$ between two  measures $\mu$ and $\nu$ in $\PP(\XX)$. Then, there always exists an optimal transport plan $\gamma_G$ that is a vertex of the transportation polytope $\Pi(\mu,\nu)$ and which satisfies
\begin{equation}\label{diag:max}
\gamma_G(x,x)=\min\{\mu(x),\nu(x)\} \text{ for all } x \in \XX.
\end{equation}
\end{prop}

\begin{proof}[Proof of Proposition \ref{prop:coldiag}]
Let $\gamma_G$ be an optimal transport plan for the OT problem  between two  measures $\mu$ and $\nu$ with ground cost $d_G$. 
Since $\gamma_G$ is an admissible transport plan, one necessarily has that  $\gamma_G(x,x) \leq \min\{\mu(x),\nu(x)\}$ for all $x \in \XX$, else the marginal constraints cannot be satisfied. Suppose that there exists $x \in \XX$ such that  $\gamma_G(x,x) < \min\{\mu(x),\nu(x)\}$. Since
 $
\sum_{y \in \XX} \gamma_G(x,y) = \mu(x)
$
and
$
\sum_{y \in \XX} \gamma_G(y,x) = \nu(x),
$
there necessarily exist $y_1 \neq x$ and $y_2 \neq x$ such that $\gamma_G(x,y_1) > 0$ and $\gamma_G(y_2,x) > 0$. Now, let
$$
m = \min\left(\gamma_G(x,y_1),\gamma_G(y_2,x),\delta_x \right) \quad \text{with} \quad \delta_x = \min\{\mu(x),\nu(x)\} - \gamma_G(x,x),
$$
and construct the plan $\tilde{\gamma}_G$ having all its entries equal to those of $\gamma_G$ except for the following ones:
\begin{equation} \label{eq:gamma_tilde}
\left\{
\begin{aligned}
\tilde{\gamma}_G(x,y_1) & =  \gamma_G(x,y_1) - m, \\
\tilde{\gamma}_G(y_2,x) & =  \gamma_G(y_2,x) - m, \\
\tilde{\gamma}_G(x,x) & =  \gamma_G(x,x) + m,  \\
\tilde{\gamma}_G(y_2,y_1) & =  \gamma_G(y_2,y_1) + m.
\end{aligned}
\right.
\end{equation}
By the definition of $m$, all the entries of the plan $\tilde{\gamma}_G$ are non-negative, and it can be easily checked that it satisfies the marginal constraints  $
\sum_{y \in \XX} \tilde{\gamma}_G(u,y) = \mu(u)
$
and
$
\sum_{y \in \XX} \tilde{\gamma}_G(y,u) = \nu(u),
$
for all $u \in \XX$. Hence,  $\tilde{\gamma}_G$ is an admissible transport plan.

Then, the difference between the cost of the plans  $\gamma_G$ and  $\tilde{\gamma}_G$ has to be non-positive since $\gamma_G$ is an optimal transport plan and satisfies
\begin{eqnarray*}
 0 \geq \sum_{x, y \in \XX} d_G(x,y) \gamma_G(x,y) &-& \sum_{x, y \in \XX} d_G(x,y) \tilde{\gamma}_G(x,y)  \\
 & = & m\left(d_G(x,y_1)  +  d_G(y_2,x)  - d_G(x,x)  - d_G(y_2,y_1) \right) \\
 & = &  m\left(d_G(x,y_1)  +  d_G(y_2,x) - d_G(y_2,y_1) \right). 
\end{eqnarray*}
Combining the above inequality with the triangle inequality we thus obtain that
\begin{align}\label{eq:passage}
    d_G(x,y_1)  +  d_G(y_2,x) = d_G(y_2,y_1),
\end{align}
Consequently, the transport costs of $\tilde{\gamma}_G$ and $\gamma_G$ are equal, and  we thus deduce that $\tilde{\gamma}_G$ is also an optimal transport plan with
$$
\min\{\mu(x),\nu(x)\} \geq \tilde{\gamma}_G(x,x) > \gamma_G(x,x) 
$$
since $m= \min\left(\gamma_G(x,y_1),\gamma_G(y_2,x),\delta_x \right)>0$ since all the elements in the set where the minimum is taken are positive. Hence,as long as there exists $x  \in \XX$ with $\tilde{\gamma}_G(x,x) <  \min\{\mu(x),\nu(x)\}$, we may repeat this construction by modifying two rows and two columns of  $\tilde{\gamma}_G$ as in \eqref{eq:gamma_tilde}, and we finally obtain the existence of an optimal plan that satisfies $\tilde{\gamma}_G(x,x) = \min\{\mu(x),\nu(x)\}$ for all $x \in \XX$.

Let us now prove the existence of such an OT plan maximizing the diagonal mass which is also a vertex of the transportation polytope  $\Pi(\mu,\nu)$. To this end, we first need to recall the concept of a face in convex analysis  \cite[Section~18]{Rockafellar1970} and two equivalent characterizations for a polytope; see e.g.\ \cite[Theorem 7.5]{brondsted1983introduction}.  A nonempty subset $\FF$ of the transport polytope $\Pi(\mu,\nu)$ is called a face if and only if either of the following equivalent conditions holds
\begin{itemize}
\item[(i)] for every $\gamma_1,\gamma_2\in \Pi(\mu,\nu)$  and every $\lambda\in(0,1),$
$
(1-\lambda)\gamma_1+\lambda \gamma_2\in \FF
$
implies that
$
\gamma_1,\gamma_2\in \FF.
$
\item[(ii)] there exists $c\in\RR^{\XX \times \XX}$, with $c \neq 0$, such that
$
\FF= \argmax_{\gamma \in \Pi(\mu,\nu)}\sum_{x,y \in \XX} c(x,y) \gamma(x,y).
$
\end{itemize}
Then, recall that
\begin{equation}
K_{d_G}(\mu,\nu) = \min_{\gamma \in \Pi(\mu,\nu)} \sum_{x, y \in \XX} d_G(x,y) \gamma(x,y), \label{eq:OT_pb}
\end{equation}
and let
$$
D(\mu,\nu) =\sum_{x\in\XX}\min\{\mu(x),\nu(x)\}.
$$
By the above construction \eqref{eq:gamma_tilde}, there exists an optimal transport plan such that the sum of its diagonal elements equals $D(\mu,\nu)$. Hence, the set
\begin{equation}
\FF=
\left\{\gamma\in\Pi(\mu,\nu):
\sum_{x,y\in\XX}d_G(x,y)\gamma(x,y)=K_{d_G}(\mu,\nu) \text{ and } 
\sum_{x\in\XX}\gamma(x,x)=D(\mu,\nu)
\right\} \label{eq:F}
\end{equation}
is nonempty. Moreover, it is a face of the transportation polytope $\Pi(\mu,\nu)$.
Indeed,
let
$$
\FF_1 =
\left\{
\gamma\in\Pi(\mu,\nu):
\sum_{x,y\in\XX}d_G(x,y)\gamma(x,y)=K_{d_G}(\mu,\nu)
\right\}
$$
denote the set of OT plans. Since the transport cost is a linear function of $\gamma$, it follows that $\FF_1$
is a  face of the transportation polytope $\Pi(\mu,\nu)$. We then consider the following maximization problem among the OT plans in $\FF_1$:
\begin{equation} \label{eq:diag_pb}
\max_{\gamma\in\mathcal F_1}
\sum_{x\in\XX}\gamma(x,x)  = D(\mu,\nu).
\end{equation}
The  second equality above follows from the above construction \eqref{eq:gamma_tilde}, which establishes the existence of an OT plan satisfying
\[
\gamma(x,x)=\min\{\mu(x),\nu(x)\},
 \text{ for all }  x\in\XX,
\]
and the fact that $\gamma(x,x) \leq \min\{\mu(x),\nu(x)\}$ for any transport plan. Hence the set $\FF$ defined by \eqref{eq:F}
is the set of solutions to the maximisation problem \eqref{eq:diag_pb}. Since $\sum_{x\in\XX}\gamma(x,x)$ is obviously a linear functional of $\gamma$, it follows that $\FF$ is a face of $\FF_1$.

Now, since $\FF_1$ is a face of $\Pi(\mu,\nu)$ and $\FF$ is a face of $\FF_1$, we may conclude that $\FF$ is a face of $\Pi(\mu,\nu)$. Indeed, if $\gamma_1,\gamma_2\in\Pi(\mu,\nu)$ and
$$
\lambda \gamma_1+(1-\lambda)\gamma_2\in\mathcal F,
\qquad \lambda\in(0,1),
$$
then, since $\FF\subseteq\FF_1$ and $\FF_1$ is a face of $\Pi(\mu,\nu)$, we have that $\gamma_1,\gamma_2\in\FF_1$ by definition of a face. Furthermore, since $\FF$ is a face of $\FF_1$, it follows that $\gamma_1,\gamma_2\in\FF$. This proves that $\FF$ is a face of $\Pi(\mu,\nu)$.

By construction, $\FF$ is nonempty. Hence, being a nonempty face of a polytope, it admits at least one vertex. Finally, every vertex of a face of a polytope is also a vertex of the polytope itself. We therefore conclude that there exists a vertex $\gamma$ in $\Pi(\mu,\nu)$
which is an optimal transport plan for the cost $d_G$ and that satisfies
$$
\gamma(x,x)=\min\{\mu(x),\nu(x)\},
\qquad \text{for all } x\in\XX.
$$
This completes the proof.

\end{proof}

Notice that if two measures $\mu,\nu\in \PP(\XX)$ have full support (that is $\mu_x \neq 0$ and $\nu_x \neq 0$ for all $x \in \XX$), then the optimal transport plan $\gamma_G$ constructed in Proposition \ref{prop:coldiag} has diagonal entries that are different from zero.

\begin{lem} \label{lem:forest}
Let $\gamma$ be a vertex of $\Pi(\mu,\nu)$  maximizing the diagonal, that is with $\gamma(x,x)= \min\{\mu(x),\nu(x)\}$ for all $x \in \XX$. Consider $\overline{\GG}_\gamma:= \GG_\gamma(\XX,\XX') \cup M$ its associated bipartite forest with vertex set $\XX \cup \XX'$ made of two copies of $\XX$ and $M$ is the graph on the same set of vertices with set of edges connecting each vertex of $\XX$ with its respective copy in $\XX'$. Then, the graph with vertex set $\XX$ obtained from  $\overline{\GG}_\gamma$ by contracting every edge in $M$ is a forest in the complete graph $K_\XX$, when forgetting loops at each $x\in \XX$
\end{lem}

\begin{proof}
    Without loss of generality one can suppose that there is no $x\in \XX$, such that $\mu(x)=\nu(x)=0$, since in such a case one can reduce the set $\XX$. So, whenever $\min\{\mu(x),\nu(x)\}=0$, it means that $x$ does not belong to the support of only one of the measures.
    Now, suppose that the support of $\mu$ is $\YY  \subset \XX$ and the support of $\nu$ is $\mathcal{Z}  \subset \XX'$.
    Therefore $\GG_\gamma(\XX,\XX') = \GG_\gamma(\YY,\ZZ)$ is a forest with vertex set $\YY \cup \ZZ$. 
    We claim that we can add the vertices in $\XX\setminus\YY$, and for each one of them we add the edge  connecting it to its copy in $\ZZ$ and still have a forest. This claim is immediate as the points in $\XX\setminus\YY$ are in $\ZZ$, this implies that each vertex added adds one edge. Moreover, the resulting graph has the same number of connected components as the original forest. By symmetry, one can also add the vertices in $\XX'\setminus \ZZ$ and the edges connecting them to their copies in $\XX$. The resulting graph $\overline{G}_\gamma$ with vertices in $\XX\cup\XX'$ contains all the edges connecting each vertex of $\XX$ with its copy in $\XX'$. See Figure \ref{fig:matching} for a representation of this step of the proof.
    
    \begin{figure}[h!]
        \centering
        \includegraphics[scale=0.6]{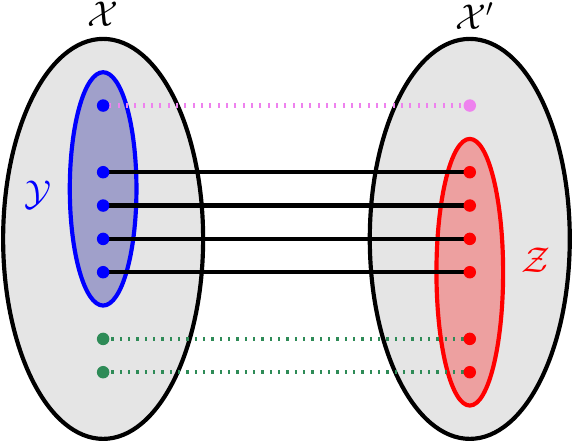}
        \caption{Representation of the edges added to $\GG_\gamma$ to obtain $\overline{G}_\gamma$. The green nodes have $\mu({\color{SeaGreen}\bullet})=0$ and the pink node has $\nu({\color{pink}\bullet})=0$. The edges added in the proof are presented as dotted lines. Here we see that each green/pink vertex adds one edge which is colored with the same color.}
        \label{fig:matching}
    \end{figure}
    
    The conclusion follows by induction on the number of connected components of $\overline{G}_\gamma$.
    Suppose that $\overline{\GG}_\gamma$ is a spanning tree, then there are $2N-1$ edges, meaning, by hypothesis that $N$ of them are in $M$, i.e. are connecting $x\in \XX$ with its copy in $\XX'$. The $N-1$ remaining edges remain present after contraction thanks to the bipartite nature of $\overline{\GG}_\gamma$ and the connectivity is preserved under contraction, so after contraction, the resulting graph is also a spanning tree (with loops in every vertex such that $\gamma(x,x)>0$; these are the edges of $M$ that were already present in $\GG_\gamma(\XX,\XX')$). The induction step is direct by contraction per connected component of $\overline{G}_\gamma$.
\end{proof}

As a direct corollary of Proposition \ref{prop:coldiag} and Lemma  \ref{lem:forest}  we obtain the following result

\begin{cor} \label{prop:forest}
Consider two measures $\mu,\nu\in \PP(\XX)$ and let $\gamma$ be a solution of the transportation problem maximizing the diagonal in the sense of \eqref{diag:max} that is a vertex of $\Pi(\mu,\nu)$. Then,  its non-oriented support  $\supp(\gamma) $ is a forest of $K_\XX$ (that is a union of trees) with at most $2N-1$ edges (including loops). As a consequence, there exist at most $N-1$ proper edges in $\supp(\gamma) $.
\end{cor}

In the rest of this sub-section, we further investigate the implications of  the above result on the structure of the non-oriented support of an OT plan. From Corollary \ref{prop:forest}, there exists an OT plan $\gamma_G$ that is a vertex of $\Pi(\mu,\nu)$ such that its non-oriented support $\supp(\gamma_G)$ is a forest (when forgetting loops).
Then,  for such $\gamma_G$, let us  denote by $\XX_1,\ldots,\XX_r$ the components of the forest $\supp(\gamma_G)$ that form a partition of the set of vertices $\XX$ such that $\gamma_G(x,y) = 0$ for all $(x,y) \in \XX_i \times \XX_j$ whenever $i \neq j$. One has that
$$
\mu(\XX_i) = \sum_{x \in \XX_i} \mu(x) = \sum_{x \in \XX_i} \sum_{y \in \XX} \gamma_G(x,y)  =  \sum_{x \in \XX_i}  \left( \sum_{j=1}^{r} \sum_{y \in \XX_j} \gamma_G(x,y) \right) = \sum_{x \in \XX_i} \sum_{y \in \XX_i} \gamma_G(x,y).
$$
Similarly, one has that $\nu(\XX_i) = \sum_{x \in \XX_i} \sum_{y \in \XX_i} \gamma_G(x,y)$, and therefore we obtain that
\begin{equation}
\sum_{x\in \XX_i}(\mu(x)-\nu(x)) = 0 \quad \mbox{for all} \quad 1 \leq i \leq r. \label{eq:components}
\end{equation}
The property \eqref{eq:components} that the measures $\mu$ and $\nu$ have equal mass on the components of the support of the optimal transport plan $\gamma_G$ is in direct relation with Definition \ref{def:weak-non-deg} on the weak  non-degeneracy of a pair of probability measures which yields the following result.%

\begin{prop} \label{prop:unique-plan}
Suppose that the pair of probability measures $\mu, \nu \in \PP(\XX)$ is weakly non-degenerate. Then, there exists an optimal transport plan $\gamma_G$ such that its non-oriented support  is a spanning tree of $K_{\XX}$ (up to loops), that is $\supp(\gamma_G)$ has exactly $N-1$ proper edges.
\end{prop}

\begin{proof}
From Corollary \ref{prop:forest} one has the existence of an optimal OT plan $\gamma_G$ such that its non-oriented support $\supp(\gamma_G)$  is a forest. Then, let us suppose it is not connected. By property \eqref{eq:components}, there exists a proper subset $\XX'$ of $\XX$ satisfying $\mu(\XX') = \nu(\XX')$ which contradicts the weak non-degeneracy of $(\mu, \nu)$. Therefore the non-oriented support of $\gamma_G$ is necessarily a tree (up to loops).
\end{proof}

\subsection{Kantorovich duality}

We recall below some well-known properties of Kantorovich potentials in the OT literature.

\begin{rem} \label{rmq:Lip}
It holds that $u$ belongs to $ \Lip$ if, and only if (see e.g.\ \cite{pistone22}[Proposition 5]),
\begin{equation}
|u(x) - u(y)| \leq d_G(x,y) \quad \mbox{for all} \quad \{x,y\} \in \ee_G \label{eq:cond_Lip}
\end{equation}
Hence, in the optimisation problem \eqref{eq:dual_Kant}, the constraint that $u \in \Lip$ can be replaced by the condition \eqref{eq:cond_Lip}.
\end{rem}

Now, by the so-called complementary slackness between the primal Kantorovich problem \eqref{eq:OT_Kant} and its dual formulation \eqref{eq:dual_Kant}, there exists a key relationship between any optimal transport plan and any optimal  dual potential that is recalled below, see e.g.\ \cite{peyre2020computational}[Section 3.3 and Section 6].

\begin{prop}\label{prop:complementary}
Let $\gamma_G$ be any optimal transport plan and $u_G$ be any optimal  dual potential. If $\gamma_G(x,y) > 0$ then necessarily $u_G(x) - u_G(y) = d_G(x,y)$, and if   $u_G(x) - u_G(y) < d_G(x,y)$ then necessarily $\gamma_G(x,y) = 0$. Conversely, if $\gamma \in \Pi(\mu,\nu)$ and $u \in  \Lip$ are complementary with respect to $d_G$ in the sense that
\begin{equation}
\mbox{  for all $(x,y) \in \XX \times \XX$,} \quad \gamma(x,y) > 0  \quad \mbox{implies that} \quad u(x) - u(y) = d_G(x,y), \label{eq:complementary}
\end{equation}
then $\gamma$ is  an optimal transport plan and $u$ is an optimal  dual potential.
\end{prop}

\subsection{Definition of the Arens-Eells norm}

Let us define the vector space $M_{0}(\XX)$ as the set of all 0-mass functions $\xi$ on $\XX$, that is, verifying $\sum_{x \in \XX} \xi(x) = 0$. Given  $\mu$ and $\nu$ in $\PP(\XX)$, an example of a key function belonging to $M_{0}(\XX)$ is $\xi = \mu - \nu$ in what follows.  Introducing the notation $\delta_{x}$ for the function on $\XX$ such that $\delta_{x}(y) = 0$ if $y \neq x$ and $1$ otherwise, it follows that the differences of delta functions span the whole space of 0-mass functions. This means that every function $\xi$ in $M_{0}(\XX)$ can be written as
\begin{equation}
\xi = \sum_{x,y \in \XX} a(x,y) (\delta_{x}-\delta_{y}), \label{eqref:xi_a}
\end{equation}
for some set of coefficients $a \in \RR^{\XX \times \XX}$. Obviously, there exist many ways to find a matrix $a$ of real coefficients verifying the representation \eqref{eqref:xi_a} of $\xi$, and this representation can also be expressed as
$$
\xi(x) = \sum_{y \in \XX} a(x,y) - \sum_{y \in \XX} a(y,x).
$$
Following  \cite{Weaver18}, and as also explained in \cite{pistone22}, the vector space $M_{0}(\XX)$ can be endowed with the following norm, and the resulting normed vector space is called the Arens-Eells space.

\begin{definition}\label{def:AEnorm}
Let $d$  be a  distance on $\XX$. For a function $\xi \in M_{0}(\XX)$, its Arens-Eells norm is defined by
\begin{equation}
\| \xi \|_{\ArEl(d)} = \inf \left\{ \sum_{x,y \in \XX} |a(x,y)| d(x,y) \right\}, \label{eq:AE}
\end{equation}
where the above infimum is taken over all representation of $\xi$ by coefficients $a \in \RR^{\XX \times \XX}$ in \eqref{eqref:xi_a}.
\end{definition}
As explained in  \cite{pistone22}[Section 2.1],  using the dual formulation \eqref{eq:dual_Kant} of OT  combined with a dual expression of Arens-Eells (AE) norms, the fundamental connection between the K-distance and the  AE norm is that
\begin{equation}
K_{d}(\mu,\nu) = \| \mu - \nu \|_{\ArEl(d)}, \label{eq:equiv_K_AE}
\end{equation}
for any probability measures $\mu$ and $\nu$ in $\PP(\XX)$ and distance $d$ on $\XX$.

\subsection{The special case of a distance induced by a tree} \label{sec:tree}

Given Definition \ref{def:AEnorm}, computing the AE-norm is generally a delicate optimisation problem. Nevertheless, when the distance  on $\XX$ is induced by a tree, the AE-norm has a closed-form expression.

 To be more precise, we first recall some basic properties and notation for trees.  Let $T = (\XX, \ee_T, w)$ be a weighted tree, that is, a connected graph without cycles with elements of $\XX$ as vertices. Throughout the paper, it is  assumed that the edges of $T$ are weighted by the symmetric function  $w : \XX \times \XX \to \RR^+$ defined previously.
 Then, as before, for such a weighted tree, the distance $d_T(x,y)$ between two points $(x,y)$ is the $w$-weighted length of the (unique) path connecting them. Notice that in particular if $\{x,y\}\in \ee_T$, then $d_T(x,y)=w(x,y)$.
 
 For our purposes it will be useful to assign a canonical orientation to the trees we will work with. To this end, we will consider a distinguished vertex $r$ that we refer as the root and the (canonical) orientation is defined in terms of the root as follows: each edge $\{x,y\}$ will be oriented going towards the root $r$, with orientation to be denoted as $y \to x$ if $x$ is the point that is the closest to the root in the sense that $d_T(x,r)< d_T(y,r)$. The property that will be useful for us about this orientation is that every vertex, but the root, has exactly one outgoing edge and the graph is connected if and only if the directed graph is a rooted tree. Given a vertex $x$ in a rooted tree, the set $\child(x)$ of children of $x$  is defined as the set of all vertices $y$ such that $y \to x$, while the unique parent of $x \neq r$ is denoted by $x^{+}$. A rooted tree thus induces a partial order between vertices that we denote as $y \lesssim x$, including $x$ as a descendant of itself, which reads as $y$ is a descendant of $x$ or $x$ is an ancestor of $y$. For a given vertex $x$ in a rooted tree $T$, we denote by $T_x$  the subtree of $T$, rooted at $x$, that contains all the vertices $y$ such that  $y  \lesssim x$. Finally, for every $(x,y)\in \XX \times \XX$, we denote by $P_{x,y} \subset \ee_T$ the unique path going from $x$ to $y$ in the tree $T$
 
 Then, the following important property holds for  the AE-norm  when the distance $d$ is induced by a tree.

 \begin{prop}[Theorem 4 and Theorem 7 in \cite{pistone22}]\label{prop:AE_tree}
 Let $T = (\XX, \ee_T,w, r)$  be a weighted tree rooted at $r$ and suppose that $d_T$ is the distance on the tree $T$ induced by the weight function $w$. Then, for any $\xi \in M_{0}(\XX)$,
 $$
 \| \xi \|_{\ArEl(T)}:= \| \xi \|_{\ArEl(d_T)} = \sum_{x \in \XX \setminus \{r\}} d_T(x,x^{+}) |\Xi_T(x)|
 $$
 where for $x \in \XX$ the quantity $\Xi_T(x)$ is the cumulative function defined as
 \begin{equation}
 \Xi_T(x) = \sum_{y \lesssim x} \xi(y). \label{eq:xi}
\end{equation}
 Moreover, one has that 
 \begin{equation}
 \| \xi \|_{\ArEl(d_T)} = \sum_{y \in \XX} u_{T}(y) \xi(y) , \label{eq:ubar}
\end{equation}
where
\begin{align}\label{eq:pot}
u_{T}(y) = \sum_{y \lesssim x \neq r} d_{T}(x,x^{+}) \sgn ( \Xi_T(x) ) \quad \mbox{ for all } y \in \XX,
\end{align}%
with $\sgn ( \cdot ) \in \{-1,+1\}$ the usual sign function having value at zero either $\sgn ( 0 ) = +1$ or $\sgn ( 0 ) = -1$.
 \end{prop}

Hence, when the distance $d = d_T$ on $\XX$ is induced by a rooted tree, the AE-norm has a simple closed-form expression which only depends  on the tree structure and not on the optimal plan chosen, and as a byproduct for the K-distance $K_{d_T}(\mu,\nu)$ by taking $\xi = \mu -\nu$ as shown by Equality \eqref{eq:equiv_K_AE}. In this setting, such a closed-form expression for the K-distance is also known in the OT literature. It has  been initially derived in \cite{EM12}, and we refer to \cite{MV23} for a detailed course on the  K-distance when the metric is induced by trees. For applications in machine learning on the use of optimal transport on trees, we also refer to \cite{Tam19,Sato20,Chen24}. 

For a probability measure $\mu \in \PP(\XX)$ and a rooted tree $T$, we define the notation
\begin{align}
    \mu(T_x) = \sum_{y \lesssim x} \mu(y).
\end{align}
Then, combining Equality \eqref{eq:equiv_K_AE} with Proposition \ref{prop:AE_tree}, we obtain that the K-distance, associated to a rooted tree, has the following expression
\begin{equation} \label{eq:K-dist_tree1}
K_{d_T}(\mu,\nu) = \sum_{x \in \XX \setminus \{r\}} d_T(x,x^{+}) \left| \mu(T_x) - \nu(T_x)\right|. 
\end{equation}
As a consequence of this formula we obtain that the K-distance is invariant when replacing $\mu(T_x)$ by $\mu(\XX\setminus T_x) := 1-\mu(T_x)$ (this is the sum of $\mu$ over the vertices not belonging to $T_x$)
at any $x$ in the interior of the sum since $\left| \mu(T_x) - \nu(T_x)\right| = \left| \mu(\XX\setminus T_x) - \nu(\XX\setminus T_x)\right|$ and this implies that $K_{d_T}(\mu,\nu)$ is invariant with respect to the choice of the root.

Finally, for the function $u_{T}$ defined  by Equation \eqref{eq:pot}, it is not difficult to check that it belongs to the set $\LipT$. Indeed, let $\{y_1,y_2\} \in \ee_T$ with $y_1 \lesssim y_2$ that is $y_2 = y_1^{+}$. Then, using the decomposition
\begin{eqnarray*}
u_{T}(y_1) & = & d_{T}(y_1,y_1^{+}) \sgn ( \Xi_T(y_1) ) + \sum_{y_2 \lesssim x \neq r} d_{T}(x,x^{+}) \sgn ( \Xi_T(x) ) \\
& = & d_{T}(y_1,y_1^{+}) \sgn ( \Xi_T(y_1) ) + u_{T}(y_2),
\end{eqnarray*}
one has that $u_{T}(y_1) -  u_{T}(y_2) = \pm d_{T}(y_1,y_1^{+})$ and thus, since $y_2 = y_1^{+}$,
\begin{equation}
|u_{T}(y_1) -  u_{T}(y_2)| = d_{T}(y_1,y_2) = w(y_1,y_2) \quad \mbox{ for all } \quad \{y_1,y_2\} \in \ee_T.
\end{equation}
Consequently, using Remark \ref{rmq:Lip}, we have finally proved the following proposition.
\begin{prop} \label{prop:ubar}
The function $u_{T}$  defined  by Equation \eqref{eq:pot} belongs to $\LipT$, and it satisfies
 $$
 u_{T}(y) -  u_{T}(y^{+}) = d_{T}(y,y^{+}) \sgn ( \Xi_T(y) ), \; \mbox{ for all } y \in \XX \setminus \{r\}.
 $$ 
 \end{prop}
Now, using Proposition \ref{prop:AE_tree} with $\xi = \mu -\nu$, it follows that
\begin{equation} \label{eq:K-dist_tree2}
K_{d_T}(\mu,\nu) =  \sum_{y \in \XX} u_{T}(y) (\mu(y) - \nu(y)) .
\end{equation}
Therefore, $u_{T} \in \LipT$ is a maximiser  of the dual Kantorovich problem \eqref{eq:dual_Kant}, meaning that $u_{T}$ is an optimal dual potential for the OT problem with $d_T$ as the ground cost. Notice that it is imposed that $u_{T}(r)=0$. This comes from the fact that  $u_{T} \in \LipT$ has an additive constant as a degree of freedom, meaning that every shift of $u_T$ by a constant value, also gives function belonging to $\LipT$ and thus another optimal dual potential.

  \begin{rem}\label{rem:xi}
 As discussed in \cite{pistone22}, there may exist multiple solutions $u_T$ of the form \eqref{eq:pot} that satisfy Equality \eqref{eq:ubar} due to the  indeterminacy of the sign function for the vertices $x$ at which $ \Xi_T(x)$ vanishes. However, when the measures $\mu$ and $\nu$ are weakly non-degenerate 
 and taking $\xi = \mu-\nu$, it follows that $\Xi_T(x) \neq 0$ for any $x\in \XX  \setminus \{r\}$. Indeed, suppose by contradiction that $\Xi_T(x)=0$ for $x \neq r$, then
\begin{align*}
0 = \Xi_T(x) = \sum_{y\lesssim x}(\mu(y)-\nu(y)) = \mu(T_x)-\nu(T_x),
\end{align*}
which contradicts the weak non-degeneracy of $\mu$ and $\nu$. Hence, under this non-degeneracy condition, Theorem \ref{theo:unique-potential} and Proposition \ref{prop:AE_tree}  imply  there exists  a unique Kantorovich potential $u_{T}$ satisfying  $u_{T}(r) = 0$ whose expression  is given by \eqref{eq:pot}.   It is then easy to see that, when $\mu$ and $\nu$ are non-degenerate,  changing the root of the tree shifts the optimal potential \eqref{eq:pot} by a constant.
 \end{rem}

\subsection{Optimization over spanning trees}

We now provide further details on the minimization problem \eqref{min:ST} over spanning trees to compute the K-distance.
In this section, we assume that $\XX$ is a finite metric space with a distance $d_G$ that is induced by an edge-weighted graph $G = (\XX,\ee_G,w)$ assumed to be connected and without loops. We denote by $\ST$ the set of rooted spanning trees of $G$ that is the set of subgraphs of $G$ that are rooted trees containing all the vertices of $G$ and whose edges belong to $\ee_G$. A tree in $\ST$ rooted at $r$ is written as $T = (\XX,\ee_T,w,r)$, and a key remark is that, for any $\{x,y\} \in \ee_T \subset \ee_G$,
\begin{equation}
d_{T}(x,y) = w(x,y).
\end{equation}
Then, we have the following fundamental property.
\begin{prop}\label{prop:optimST} \cite{pistone22}[Theorem 10] Let  $\mu$ and $\nu$ be  two  probability measures in $\PP(\XX)$.
 For the ground cost $d_G$, the K-distance between two probability measures $\mu$ and $\nu$ on $\XX$ is the value of the following optimization problem:
\begin{eqnarray}
K_{d_G}(\mu,\nu)  &=& \min_{T \in \ST} K_{d_T}(\mu,\nu) \nonumber \\
&=&  \min_{T \in \ST}   \| \mu - \nu \|_{\ArEl(T)} \nonumber \\
&=&  \min_{T \in \ST} \sum_{x \in \XX\setminus \{r\}} d_T(x,x^{+}) \left| \mu(T_x) - \nu(T_x)\right|,\label{eq:optimST}
\end{eqnarray}
where the above minimum is taken over the set  of weighted rooted spanning trees of $G$, with $d_T$ denoting the metric induced  by a given tree $T = (\XX, \ee_T,w, r) \in \ST$.
\end{prop}

\begin{rem}
The right statement of \cite{pistone22}[Theorem 10] is to take the minimum over the set of spanning forests of the graph, but we can reduce the optimization problem to spanning trees by adding transfers with zero mass  between components of the forest. From  Cayley's formula one obtains that the number of rooted spanning trees of the complete graph on $N$ vertices is $N^{N-1}$, while the number of rooted spanning forests of the same graph is $(N+1)^{N-1}$. This gives a (small) reduction of the problem since $(N+1)^{N-1}\sim e N^{N-1}$ as $N\to\infty$.
\end{rem}

\begin{rem}
In the equalities of Proposition \ref{prop:optimST} the root of the tree is only needed in order to obtain the closed form formula in the last line \eqref{eq:optimST}, but since it is invariant under rerooting, one can use any root.
\end{rem}

\begin{rem}\label{rem:supps}
It is important to disambiguate an optimal spanning tree $T_\ast$ of $G$ minimizing \eqref{min:ST} from the support of the optimal plan $\gamma_{T_\ast}$ that is a vertex (maximising the diagonal) of $\Pi(\mu,\nu)$ to which it is associated when considering optimal transport with ground cost $d_{T_\ast}$, which is an forest of $K_{\XX}$.  There may exist $(x,y)\notin \vec{\ee}_G$ with $\gamma_{T_\ast}(x,y)>0$, meaning that a positive quantity of mass is sent from $x$ to $y$ using $P_{x,y}$ in $T_\ast$, as illustrated in  Figure \ref{fig:supports}.
We claim that if the optimal plan $\gamma_{T_\ast}$ associated to $K_{d_{T_\ast}}(\mu,\nu)$ is a spanning tree, then $\gamma_{T_\ast}$ uses every edge of $T_\ast$ (in some direction) to send mass. To prove the claim notice that if an edge $\{x,y\}\in \ee_T$ is not used in any direction, the flow of mass is restricted to $T_x$ and $\XX\setminus T_x$, meaning that the support of $\gamma_{T_\ast}$ is a forest and not a tree.
\begin{figure}[htbp]
    \centering
    \includegraphics{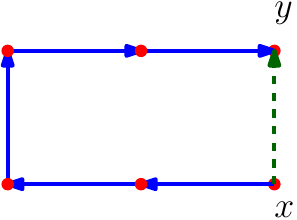}
    \caption{The red dots represent a set of $N=6$ vertices, and the blue segments are the edges in $T_\ast$. The green and dashed arrow is an oriented edge $(x,y)$ in the support  of $\gamma_{T_\ast}$ (that is such that $\gamma_{T_\ast}(x,y)>0$) which does not belong to $\ee_{T_\ast}$. The mass sent from $x$ to $y$ follows the blue path.}
    \label{fig:supports}
\end{figure}
\end{rem}

\section{Theoretical results : properties of the optimal spanning trees, plan and potentials} \label{sec:contributions}

In this section, we study several properties of the optimal spanning trees that minimize \eqref{eq:optimST}, and we study their connections to the solutions of the primal \eqref{eq:OT_Kant} and dual \eqref{eq:dual_Kant} formulations of Kantorovich OT when taking the metric $d_G$ as the ground cost. 

\subsection{From spanning trees to optimal couplings}

Let us first remark that a closed-form of an optimal coupling $\gamma_G$ between two arbitrary probability measures $\mu$ and $\nu$ in $\PP(\XX)$ is generally not available. One has thus to rely on computational approaches to solve the linear program \eqref{eq:OT_Kant} that yield a numerical approximation of an OT plan. Interestingly, thanks to the closed-form expressions  \eqref{eq:K-dist_tree1} and  \eqref{eq:K-dist_tree2}  of the K-distance  when   the ground cost is a metric induced by a tree, it is possible to find an explicit expression of  $\gamma_G$ and the related optimal dual potential $u_G$ under restrictive assumptions on $\mu$ and $\nu$ when the ground cost is the metric induced by the weighted graph $G = (\XX,\ee_G,w)$. 

The result below is partly inspired by  \cite{pistone22}[Theorem 6]. Before stating it, we need the following definitions.
For $a\in \R$ we define its positive part $a^+$ and negative part $a^-$ as
\[
	a^+ = \max(a,0) \quad \text{and} \quad   a^- = \max(-a,0).
\]
It is easy to see that at most one of them is different from $0$, $a = a^+ - a^-$ and $|a|= a^+ + a^-$.

\begin{prop}\label{prop:OT_formula} Let  $\mu$ and $\nu$ be  two  probability measures in $\PP(\XX)$. Let $T_\ast \in \ST$  be a minimizer of the optimization problem \eqref{eq:optimST} rooted at $r$. Suppose that  $T_\ast$ satisfies the following condition
\begin{equation} \label{eq:condTast}
\Xi_{T_\ast}(x)  \cdot \Xi_{T_\ast}(y)  < 0, \; \mbox{ for all }  y \in \child(x),  \; \mbox{ and } \; x \in \XX\setminus \{r\}, 
\end{equation}
where  $\Xi_{T_\ast}(x) = \sum_{y \lesssim x} \xi(y)$ for $\xi = \mu - \nu$, and $\cdot$ denotes the usual multiplication between two scalars.
Then,  if 
\begin{equation}
\mu(r) - \sum_{y \in \child(r)} \Xi_{T_\ast}^{-}(y) \geq 0, \label{eq:condmu_r}   
\end{equation}
 the mapping $\gamma_{T_\ast} : \XX \times \XX \to \RR$ defined by
\begin{eqnarray*}
\gamma_{T_\ast}(x,x^{+}) & =  & \Xi_{T_\ast}^{+}(x), \; \mbox{ for all } x\in \XX\setminus\{r\}, \\
\gamma_{T_\ast}(x^{+},x) & =  & \Xi_{T_\ast}^{-}(x), \; \mbox{ for all } x\in \XX\setminus\{r\}, \\
\gamma_{T_\ast}(x,x) & = & \mu(x) -  \Xi_{T_\ast}^{+}(x) - \sum_{y \in \child(x)} \Xi_{T_\ast}^{-}(y), \\
 \gamma_{T_\ast}(x,y) & = & 0 \mbox{ otherwise},
\end{eqnarray*}
is  an admissible transport plan between $\mu$ and $\nu$, that is $\gamma_{T_\ast} \in \Pi(\mu,\nu)$.
Moreover, define the function
\begin{equation} \label{eq:uast}
u_{T_\ast}(y) = \sum_{y \lesssim x \neq r} d_{T_\ast}(x,x^{+}) \sgn ( \Xi_{T_\ast}(x) ), \quad \mbox{ for all } y \in \XX,
\end{equation} 
and suppose that it satisfies the condition
\begin{equation} \label{eq:cond-uast}
u_{T_\ast} \in \Lip.
\end{equation} 
Then, one has that $\gamma_{T_\ast}$ is an optimal coupling for the ground cost $d_G$ and
$u_{T_\ast}$ is an optimal dual potential, that is
$$
K_{d_G}(\mu,\nu)  =  \sum_{x, y \in \XX} d_G(x,y) \gamma_{T_\ast}(x,y) =  \sum_{y \in \XX} u_{T_\ast}(y) (\mu(y) - \nu(y)).
$$%
\end{prop}

\begin{proof}[Proof of Proposition \ref{prop:OT_formula}]
Let us first prove that $\gamma_{T_\ast}$ is feasible, that is $\gamma_{T_\ast} \in \Pi(\mu,\nu)$.  Using the definition of $\gamma_\ast$, for $x \in \XX\setminus \{r\}$, one has that
    \begin{align*}
        \sum_{y\in \XX}\gamma_{T_\ast}(x,y)&=\gamma_{T_\ast}(x,x)+\gamma_{T_\ast}(x,x^+)+\sum_{u \in \child(x)} \gamma_{T_\ast}(x,u)\\
        &= \gamma_{T_\ast}(x,x)+\Xi_{T_\ast}^{+}(x)+\sum_{u \in \child(x)} \Xi_{T_\ast}^{-}(u) = \mu(x),
    \end{align*}
and for $x=r$,
$$
\sum_{y\in \XX}\gamma_{T_\ast}(r,y)=\gamma_{T_\ast}(r,r)+\sum_{y \in \child(r)} \gamma_{T_\ast}(r,y) = \mu(r).
$$

Then, using the fact that
$
\xi(y) = \Xi_{T_\ast}(y) - \sum_{u \in \child(y)} \Xi_{T_\ast}(u),
$
it follows that, for $y \in \XX\setminus \{r\}$,
    \begin{align*}
        \sum_{x\in \XX}\gamma_{T_\ast}(x,y)&=\gamma_{T_\ast}(y,y)+\gamma_{T_\ast}(y^+,y)+\sum_{u \in \child(y)} \gamma_{T_\ast}(u,y)\\
        &= \gamma_{T_\ast}(y,y)+\Xi_{T_\ast}^{-}(y)+\sum_{u \in \child(y)} \Xi_{T_\ast}^{+}(u)\\
        &= \mu(y)-\Xi_{T_\ast}^{+}(y)-\sum_{u \in \child(y)} \Xi_{T_\ast}^{-}(u) +\Xi_{T_\ast}^{-}(y)+\sum_{u \in \child(y)} \Xi_{T_\ast}^{+}(u)\\
        &=\mu(y)-\Xi_{T_\ast}(y)+\sum_{u \in \child(y)} \Xi_{T_\ast}(u)\\
        &=\mu(y)-\xi(y)=\nu(y),
    \end{align*}
and for $y = r$,
\begin{align*}
        \sum_{x\in \XX}\gamma_{T_\ast}(x,r)&=\gamma_{T_\ast}(r,r)+\sum_{u \in \child(r)} \gamma_{T_\ast}(u,r)\\
        &= \mu(r) -\sum_{u \in \child(r)} \Xi_{T_\ast}^{-}(u) +\sum_{u \in \child(r)} \Xi_{T_\ast}^{+}(u)\\
        &=\mu(r) +\sum_{u \in \child(r)} \Xi_{T_\ast}(u)\\
        &=\mu(r)-\xi(r)=\nu(r).
    \end{align*}

Obviously, by construction, one has that $\gamma_{T_\ast}(x,y)  \geq 0$ for all $(x,y) \in \XX \times \XX$ with $x \neq y$. Therefore, thanks to the above computations, it only remains to prove that $\gamma_{T_\ast}(x,x) \geq 0$ for all $x \in \XX$   to conclude that $\gamma_{T_\ast} \in \Pi(\mu,\nu)$. First, if $x=r$ then $\Xi_{T_\ast}^{+}(r) = 0$ and thus, by construction,
$$
\gamma_{T_\ast}(r,r)  =  \mu(r) - \sum_{y \in \child(r)} \Xi_{T_\ast}^{-}(y) \geq 0
$$
by Condition \eqref{eq:condmu_r}.
 Now, let $x \in \XX\setminus \{r\}$ and suppose that $\Xi_{T_\ast}(x)> 0$. By condition \eqref{eq:condTast}, one has thus that $ \sgn( \Xi_{T_\ast}(y) ) < 0$ for all $y \in \child(x)$. Now, using the equality
\begin{equation}
\xi(x) = \Xi_{T_\ast}(x) - \sum_{y \in \child(x)} \Xi_{T_\ast}(y) \label{eq:xiXi}
\end{equation}
combined with the facts that $ \Xi_{T_\ast}(x) =  \Xi_{T_\ast}^{+}(x) $ and $\Xi_{T_\ast}(y) = -\Xi_{T_\ast}^{-}(y)$ for  all $y \in \child(x)$, 
we obtain that
$$
\mu(x) -  \Xi_{T_\ast}^{+}(x) = \nu(x) + \sum_{y \in \child(x)} \Xi_{T_\ast}^{-}(y).
$$
 This implies that
$$
\gamma_{T_\ast}(x,x) = \mu(x) -  \Xi_{T_\ast}^{+}(x)  - \sum_{y \in \child(x)} \Xi_{T_\ast}^{-}(y) = \nu(x) \geq 0.
$$
Now, the fact that $\Xi_{T_\ast}(x) \leq  0$ implies, by condition \eqref{eq:condTast},   that $ \sgn( \Xi_{T_\ast}(y) ) > 0$ for all $y \in \child(x)$. Therefore, by definition of $\gamma_{T_\ast}(x,x)$, this yields
$$
\gamma_{T_\ast}(x,x) = \mu(x) -  \Xi_{T_\ast}^{+}(x)  - \sum_{y \in \child(x)} \Xi_{T_\ast}^{-}(y) = \mu(x) \geq 0.
$$
We have thus proved that $\gamma_{T_\ast}(x,x) \geq 0$ for all $x \in \XX$ implying that $\gamma_{T_\ast} \in \Pi(\mu,\nu)$.

Let us now prove that $\gamma_{T_\ast}$ and $u_{T_\ast}$ are complementary in the sense of \eqref{eq:complementary}. By construction of $\gamma_{T_\ast}$, its support is restricted to the pairs $\{x,x\}$, $\{x,x^+\}$ and $\{x^+,x\}$ of edges in $\ee_{T_\ast}$. We first remark that $\gamma_{T_\ast}(x,x) > 0$ necessarily implies that $u_{T_\ast}(x) - u_{T_\ast}(x) = 0 = d_{G}(x,x)$. Now, for $x \in \XX\setminus \{r\}$, one has that $\gamma_{T_\ast}(x,x^+) > 0$ implies that $\sgn(\Xi_{T_\ast}(x)) > 0$, and thus by Proposition \ref{prop:ubar} one has that $u_{T_\ast}(x) -  u_{T_\ast}(x^{+}) = d_{T_\ast}(x,x^{+}) = w(x,x^{+}) = d_{G}(x,x^{+})$. Similarly,  one has that $\gamma_{T_\ast}(x^+,x) > 0$ implies that $\sgn(\Xi_{T_\ast}(x)) < 0$, and thus by Proposition \ref{prop:ubar} one has that $u_{T_\ast}(x) -  u_{T_\ast}(x^{+}) = -d_{T_{\ast}}(x,x^{+}) = -w(x,x^{+}) = -d_{G}(x,x^{+})$, meaning that $u_{T_\ast}(x^{+}) -  u_{T_\ast}(x) = d_{G}(x^{+},x)$ by symmetry of a distance. Hence, we have proved that $\gamma_{T_\ast}$ and $u_{T_\ast}$ are complementary.

Now, since $u_{T_\ast} \in \Lip $ (by assumption) it follows from Proposition \ref{prop:complementary} that  $\gamma_{T_\ast}$ is an optimal coupling for the ground cost $d_G$ and $u_{T_\ast}$ is an optimal dual potential which completes the proof of Proposition \ref{prop:OT_formula}.
\end{proof}

\begin{rem}
The first part of Proposition \ref{prop:OT_formula}  on a closed form expression for an optimal transport plan  is due to   \cite{pistone22}[Theorem 6]. Under the condition \eqref{eq:condTast}, this expression of $\gamma_{T_\ast}$ as an optimal transport plan associated to the $K$-distance with ground cost $d_T$ is valid for any tree $T$. 
\end{rem}

The novelty of  Proposition \ref{prop:OT_formula} is to prove that, under the Lipschitz condition \eqref{eq:cond-uast}, the function $u_{T_\ast}$ is a Kantorovich potential for the $K$-distance associated to the ground cost $d_G$. More precisely,  if $T_\ast \in \ST$  is a minimiser of the optimisation problem \eqref{eq:optimST}, it follows from Equality \eqref{eq:K-dist_tree2} that
$$
K_{d_G}(\mu,\nu)  = K_{d_{T_\ast}}(\mu,\nu) =  \sum_{y \in \XX} u_{T_\ast}(y) (\mu(y) - \nu(y)).
$$
Then, if the following two conditions hold: $T_\ast$ satisfies \eqref{eq:condTast} and  $u_{T_\ast} \in \Lip $, then Proposition \ref{prop:OT_formula} proves how one obtains an optimal transport plan (in a closed form) for  OT with ground cost $d_G$ by solving the optimisation problem  \eqref{eq:optimST} on the set $\ST$ of spanning trees of $G$. We next discuss these two conditions in more detail. Before doing so, we justify in the  next section that condition \eqref{eq:condTast} is exactly the setup where an optimal transport plan can be directly built from the solution of a minimum cost flow problem called the Beckmann formulation of OT. In general, as we will see, a transport plan is obtained from a path decomposition of a flow solution. Here condition  \eqref{eq:condTast} says that this decomposition is trivial so that the flow is directly its path decomposition.

\subsubsection{The connection with the Beckmann formulation of OT on a graph}

	It is well known in the OT literature that the transportation problem using the ground cost $d_G$ induced by a graph $G$ is intimately related to the minimum cost flow problem, see e.g.\ \cite{peyre2020computational}[Chapter 6], and the so-called Beckmann formulation  \cite{BF}  of OT that reads as:
	\begin{equation} \label{eq:beck}
		K_{d_G}(\mu,\nu)  = \min_{f \in \RR_+^{\Vec{\ee}_G}} \left\{ \sum_{(x,y) \in \Vec{\ee}_G} w(x,y) f(x,y) \; : \; \text{div}(f) = \mu- \nu \right\},
	\end{equation}
	where the divergence operator  $\text{div} : \RR_+^{\Vec{\ee}_G} \to \RR^\XX$ of the flow $f$ is defined on the set $\Vec{\ee}_G$ of oriented edges of $G$ as: 
	$$
	\text{div}(f)(x) = \sum_{y : (x,y) \in \Vec{\ee}_G} f(x,y) - \sum_{y : (y,x) \in \Vec{\ee}_G}  f(y,x), \text{ for all } x \in \XX.
	$$
	As explained in \cite{peyre2020computational}[Chapter 6], the  minimization problem \eqref{eq:beck} is derived from the dual formulation of the OT problem \eqref{eq:dual_Kant} when restricted to the Lipschitz constraints~\eqref{eq:cond_Lip} on the edges of $G$ combined with duality arguments. 
	
	When $G = T$ is a rooted  tree with orientation  towards the root, the set of  oriented edges of $T$ is $\Vec{\ee}_T = \{ (x,x^+), \; x \in \XX \setminus \{r\} \} \cup \{ (x^+,x), \; x \in \XX \setminus \{r\} \}$. In this framework, it is not difficult to see that there exists an explicit solution to the minimization problem \eqref{eq:beck}
	\begin{prop}\label{prop:flow_formula}   Let  $\mu$ and $\nu$ be  two  probability measures in $\PP(\XX)$. Let $G=T$ be a rooted tree and consider OT with ground cost $d_T$. A solution $f_T$ of the Beckmann problem \eqref{eq:beck} (i.e. an optimal flow) is given by
		\begin{eqnarray*}
			f_{T}(x,x^{+}) & =  & \Xi_{T}^{+}(x), \\
			f_{T}(x^{+},x) & =  & \Xi_{T}^{-}(x),
		\end{eqnarray*}
		for every $x\in \XX\setminus \{r\}$.
	\end{prop}
	\begin{proof}
		The flow $f_T$ is by definition non-negative and by definition of $\Xi_T$ and since $\Xi_T(x) = \Xi^+_T(x) -  \Xi^-_T(x)$,   one has that,  for all $x \in \XX  \setminus \{r\}$,
		\begin{eqnarray*}
			\text{div}(f_T)(x) & = & f_T(x,x^+) +  \sum_{y \in  \child(x)} f_T(y^+,y) - f_T(x^+,x) - \sum_{y \in  \child(x)} f_T(y,y^+) \\
			& = &  \Xi^+_T(x) -  \Xi^-_T(x) -  \sum_{y \in  \child(x)}  \left(\Xi^+_T(y) -  \Xi^-_T(y)  \right)\\
			& =  & \Xi_T(x) - \sum_{y \in  \child(x)}\Xi_T(y) = \xi(x) = \mu(x) - \nu(x).  
		\end{eqnarray*}
		Moreover, using the fact that $\Xi_T(r) = 0$, one also has that
		\begin{eqnarray*}
			\text{div}(f_T)(r) = \sum_{y \in  \child(r)} \left(f_T(y^+,y)  -  f_T(y,y^+)\right) = \Xi_T(r) - \sum_{y \in  \child(r)} \Xi_T(y) =\mu(r) - \nu(r).
		\end{eqnarray*}
		Therefore, the divergence of $f_T$ satisfies the constraint $\text{div}(f_T) = \mu- \nu$. Finally, since the distance $d_T$ is symmetric and given that $|\Xi_T(x) | = \Xi^+_T(x)+ \Xi^-_T(x) $, one has that
		\begin{eqnarray*}
			\sum_{(x,y) \in \vec{\ee}_T} d_T(x,y) f_T(x,y)  & = &  \sum_{x \in \XX \setminus \{r\}} d_T(x,x^+) f_T(x,x^+) +   \sum_{x \in \XX \setminus \{r\}} d_T(x^+,x) f_T(x^+,x) \\
			& = &  \sum_{x \in \XX \setminus \{r\}} d_T(x,x^+) (\Xi^+_T(x)  + \Xi^-_T(x)  ) \\
			& = &  \sum_{x \in \XX \setminus \{r\}} d_T(x,x^+) |\Xi_T(x) | = K_{d_T}(\mu,\nu), 
		\end{eqnarray*}
		which proves the assertion.
	\end{proof}
	
	\begin{rem}
		By similar arguments, when $G=T$ is a rooted tree, a solution of the minimization problem \eqref{eq:AE} involved in the Definition \ref{def:AEnorm} of the Arens-Eells norm of $\xi = \mu - \nu$ is given by
		$$
		a(x,x^+) = \Xi_T(x) \text{ and } a(x,y) = 0 \text{ when } y \neq x^+,
		$$
		for all $x \in \XX \setminus \{r\}$.
	\end{rem}

	Comparing the expression of the optimal flow $f_T$ in Proposition \ref{prop:flow_formula} for the Beckmann problem with the one of the optimal coupling $\gamma_T$ given in Proposition \ref{prop:OT_formula}, we see that they coincide outside of the diagonal of $\gamma_T$. Therefore, under the restrictive Condition \eqref{eq:condTast}, a solution of the Beckmann problem when $G=T$ immediately provides the values of the off-diagonal entries of an optimal transport plan. But this is restrictive since  transport plans naturally translate into flows but they are generally  not  equal. More formally, observe that a transport plan is indexed by $\XX \times \XX$,
which can be interpreted as the set of oriented edges of the complete graph
$K_{\XX}$, whereas a flow is indexed by the set $\vec{\ee}_{G}$ of oriented
edges of the graph $G$. From a given feasible plan $\gamma$, one can create a flow $f_\gamma$, in the case of a tree $T$, by making the sum over all the mass sent over an oriented edge in this direction by $\gamma$, i.e. for every oriented edge $\vec{e} = (u,v)$ define:
\begin{equation}
f_\gamma(u,v) = \sum_{(x,y) : (u,v) \in P_{x,y}}\gamma(x,y). \label{eq:flow_gamma}
\end{equation}

\subsection{ Relationships between OT problems with ground cost either $d_G$ or $d_{T^*}$}

Recall that we denote by $T_\ast$ an optimal spanning tree in $\ST$ satisfying
$
K_{d_G}(\mu,\nu) = K_{d_{T_\ast}}(\mu,\nu),
$
which implies the following equalities
\begin{equation}
K_{d_G}(\mu,\nu) = \sum_{x, y \in \XX} d_G(x,y) \gamma_G(x,y) =  \sum_{x, y \in \XX} d_{T_\ast}(x,y) \gamma_{T_\ast}(x,y) = K_{d_{T_\ast}}(\mu,\nu), \label{equalityOTplans}
\end{equation}
and
\begin{equation}
K_{d_G}(\mu,\nu) = \sum_{x \in \XX} u_{G}(x) (\mu(x) - \nu(x)) =  \sum_{x \in \XX} u_{T_\ast}(x) (\mu(x) - \nu(x)) = K_{d_{T_\ast}}(\mu,\nu), \label{equalityOTpotentials}
\end{equation}
where $\gamma_G$ (resp.\ $\gamma_{T_\ast}$) and $u_{G}$ (resp.\ $u_{T_\ast}$) denote an optimal transport plan and an optimal dual potential for OT with ground cost $d_G$ (resp.\ $d_{T\ast}$).

We now investigate the relationships between the optimal couplings $\gamma_G$ and $\gamma_{T_\ast}$, as well as between the corresponding potentials $u_G$ and $u_{T_\ast}$. To this end, we first need to introduce the notion of a geodesic graph.

\begin{definition}
    The geodesic graph of $G$ is the spanning graph containing all the edges belonging to a path with minimum distance, called geodesic path, between two vertices of $G$.
\end{definition}  

 The geodesic graph of a given graph $G$  together with the minimum cost paths for every pair of vertices can be obtained from the Floyd-Warshall algorithm \cite{Floyd1962,Warshall1962} which takes $O(N^3)$ where $N$ is the number of vertices of the graph \cite{Cormen2009}.
 
For some specific cases of measures $\mu$ and $\nu$, Proposition \ref{prop:OT_formula} gives a way to find an optimal coupling for the K-distance with ground cost $d_G$ from the knowledge of the optimal spanning tree $T_\ast$ solving    \eqref{eq:optimST}. However, given two arbitrary measures $\mu$ and $\nu$,  it is generally not clear whether $\gamma_{T_\ast}$ in \eqref{equalityOTplans} is an optimal plan for OT with ground cost $d_G$. The following result gives a positive answer to this question.

\begin{thm}\label{thm:plan}
    Let $\mu$ and $\nu$ be two probability measures in $\PP(\XX)$. Consider $T_\ast$ a spanning tree realizing the minimum in the  optimisation problem  \eqref{eq:optimST}. Then, one has that $\gamma_{T_\ast}$ is also an optimal plan associated to OT with ground cost $d_G$ meaning that
 $$
 K_{d_G}(\mu,\nu) =  \sum_{x, y \in \XX} d_{G}(x,y) \gamma_{T_\ast}(x,y).
 $$
Moreover, for every $(x,y) \in \XX \times \XX$ with $\gamma_{T_\ast}(x,y)>0$, one has that $d_G(x,y) = d_{T_\ast}(x,y)$, meaning that  the edges with positive transport in $T_\ast$ belong to the geodesic graph of $G$.
\end{thm}

\begin{proof}
By definition of the $K$-distance, the fact that $\gamma_{T_\ast}$ is an admissible plan and Equality \eqref{equalityOTplans}, it follows that
    \begin{align*}
        0 \leq \sum_{x,y\in \XX} \gamma_{T_\ast}(x,y)d_G(x,y) - K_{d_G}(\mu,\nu) 
        &= \sum_{x,y\in \XX} \gamma_{T_\ast}(x,y)d_G(x,y) - K_{d_{T_\ast}}(\mu,\nu)\\
        &= \sum_{x,y\in \XX} \gamma_{T_\ast}(x,y)d_G(x,y) - \sum_{x,y\in \XX} \gamma_{T_\ast}(x,y)d_{T_\ast}(x,y)\\
        &= \sum_{x,y\in \XX} \gamma_{T_\ast}(x,y)\left(d_G(x,y)-d_{T_\ast}(x,y)\right).\\
    \end{align*}
Since $\gamma_{T_\ast}(x,y)\geq 0$ and $d_G(x,y)\leq d_{T_\ast}(x,y)$ (as $T_\ast$ is a spanning tree of $G$),  the above inequality implies that $ \sum_{x,y\in \XX} \gamma_{T_\ast}(x,y)\left(d_G(x,y)-d_{T_\ast}(x,y)\right) = 0$. Therefore, one obtains that $K_{d_G}(\mu,\nu) = \sum_{x,y\in \XX} \gamma_{T_\ast}(x,y)d_G(x,y)$. Moreover,  if $\gamma_{T_\ast}(x,y)>0$, one necessarily has that $d_G(x,y)-d_{T_\ast}(x,y) = 0$, meaning that the only path $P_{x,y}$ in $T_\ast$ is a geodesic of $G$, so its edges belong to the geodesic graph. 
\end{proof}

The following result complements Theorem \ref{thm:plan} for the dual formulation of the $K$-distance giving a relationship between Kantorovich potentials for optimal transport with ground cost either $d_G$ or $d_{T_\ast}$.

\begin{thm} \label{prop:uGuT}
 Let $\mu$ and $\nu$ be two probability measures in $\PP(\XX)$. Consider $T_\ast$ a spanning tree realizing the minimum in the  optimisation problem  \eqref{eq:optimST}. Then, one has that any Kantorovich potential $u_G$ belongs to $\LipTast$ and realizes $K_{d_{T_\ast}}(\mu,\nu)$ that is
 $$
 K_{d_{T_\ast}}(\mu,\nu) = \max_{u \in \LipTast} \sum_{x \in \XX} u(x) (\mu(x) - \nu(x)) = \sum_{x \in \XX} u_{G}(x) (\mu(x) - \nu(x)). 
 $$
 Moreover,  for every $(x,y) \in \XX \times \XX$ with $\gamma_{T_\ast}(x,y)>0$,  one has that 
\[
    u_G(x)-u_G(y)|=d_G(x,y)=d_{T_\ast}(x,y).
\]
\end{thm}
\begin{proof}
    Since any Kantorovich potential $u_G$ belongs to $\Lip$  one has that
    \[
        |u_G(x)-u_G(y)|\leq d_G(x,y)\quad \forall x,y\in \XX.
    \]
    But since the metrics $d_G$ and $d_{T_\ast}$ are defined from the same weight function $w$, one has that $d_G(x,y) \leq d_{T_\ast}(x,y)$ for every  $(x,y) \in \XX \times \XX$ implying that
    \[
        |u_G(x)-u_G(y)|\leq d_{T_\ast}(x,y)\quad \forall x,y\in \XX.
    \]
    meaning that $u_G$ belongs to $\LipTast$.
    Now, it follows that $u_G$ is also a Kantorovich potential of $K_{d_{T_\ast}}(\mu,\nu)$ since, from the dual formulation of the $K$-distance  $K_{d_{T_\ast}}$ and   Equality \eqref{equalityOTpotentials},  one has that
    \[
   \max_{u \in \LipTast} \sum_{x \in \XX} u(x) (\mu(x) - \nu(x)) =     K_{d_{T_\ast}}(\mu,\nu)=K_{d_G}(\mu,\nu) = \sum_{x\in \XX}u_G(x)(\mu(x)-\nu(x)).
    \]
    Finally, to prove the last statement,   we know from Theorem \ref{thm:plan} that $d_G(x,y) = d_{T_\ast}(x,y)$ whenever $\gamma_{T_\ast}(x,y)>0$. Moreover, since $\gamma_{T_\ast}$ is an optimal plan it follows from Proposition \ref{prop:complementary} that for every $x,y\in \XX$ such that $\gamma_{T_\ast}(x,y)>0$ one has
    $
        u_G(x)-u_G(y)= d_{T_\ast}(x,y) = d_G(x,y).$
\end{proof}
A positive result concerning the lifting of the Kantorovich potential $u_{T_\ast}$ with ground cost $d_{T_\ast}$ to a Kantorovich potential for the OT problem with ground cost $d_G$ is provided by the following corollary. This result is an immediate consequence of Theorem ~\ref{prop:uGuT}, since any Kantorovich potential for the OT problem with ground cost $d_G$ is also a Kantorovich potential for OT with ground cost $d_{T_\ast}$.

\begin{cor}\label{cor:un}
Let $T_\ast$ be an optimal spanning tree of $G$ achieving $K_{d_{T_\ast}}(\mu,\nu)=K_{d_{G}}(\mu,\nu)$. If the optimal potential $u_{T_\ast}$ which realizes $K_{d_{T_\ast}}(\mu,\nu)$ is unique (up to an additive constant), then necessarily the optimal potential $u_G$ of $K_{d_{G}}(\mu,\nu)$ is unique with $u_G = u_{T_\ast}$ and its explicit expression is given by
\begin{equation}\label{eq:pot_T}
 u_G(y) =  u_{T_\ast}(y) = \sum_{y \lesssim x \neq r} d_{T_\ast}(x,x^{+}) \sgn ( \Xi_{T_\ast}(x) ), \quad \mbox{ for all } y \in \XX.
\end{equation}
\end{cor}

Finally, if the pair $(\mu,\nu)$ is weakly non-degenerate, the following result is the main contribution of this paper that establishes a connection between an optimal spanning tree  minimizing~\eqref{min:ST} and the Kantorovich potential $u_G$.

\begin{thm} \label{theo:unique-potential}
Suppose that the pair of probability measures $\mu, \nu \in \PP(\XX)$ is weakly non-degenerate. Then, there exists a unique Kantorovich potential $u_G$ (up to an additive constant), whose explicit expression is given by \eqref{eq:pot_T}.
\end{thm}

\begin{proof}
Suppose that the pair of probability measures $\mu, \nu \in \PP(\XX)$ is weakly non-degenerate and that $G = T$ is a tree. Let $\gamma_T$ be the plan given by Proposition \ref{prop:unique-plan}, one has that  $\supp(\gamma_T)$ is a tree (up to loops) with exactly $N-1$ proper edges. Complementary slackness (Proposition \ref{prop:complementary}) holds for this single $\gamma_T$ and every optimal potential, so every optimal potential solves
\begin{align}\label{sys:opt}
    u(x) - u(y) = d_T(x,y) & \text{ for all }  x \neq y \text{ such that } \gamma_T(x,y) > 0,
\end{align}
which are $N-1$ linearly independent equations for $N$
unknown variables $(u(x))_{x \in \XX}$, meaning that the Kantorovich potential $u$ obtained from the knowledge of the support of $\gamma_T$ is unique up to an additive constant. We remark that the linear independence of the $N-1$ equations of the linear system \eqref{sys:opt} follows from the tree structure of  $\supp(\gamma_T)$. 

Now, from Proposition \ref{prop:ubar}, we know that $u_T$ defined by \eqref{eq:pot} is an optimal potential for the dual optimal transport problem with ground cost $d_T$, and it is thus a solution of the linear system \eqref{sys:opt} for any optimal transport plan $\gamma$ that is a vertex of $\Pi(\mu,\nu)$ by complementary slackness (Proposition \ref{prop:complementary}). Consequently, we have that $u = u_T$ (up to an additive constant). Hence, for $G = T$, one has that  $u_{T}$ is the unique Kantorovich potential associated to optimal transport with ground cost $d_{T}$ and the conclusion for general $G$  follows from Corollary \ref{cor:un}, which completes the proof.
\end{proof}

\subsection{Optimal transport plan for distances induced by a tree.}  \label{sec:OTplan}

Theorem~\ref{thm:plan} implies that constructing an optimal transport plan $\gamma_{T_\ast}$ from an optimal spanning tree $T_\ast$ yields an optimal transport plan for the Kantorovich distance associated with the ground cost $d_G$, for an arbitrary graph $G$. We therefore focus on constructing an algorithm to compute an optimal transport plan for a ground cost $d_T$ induced by a tree $T$, based on the following two principles:
\begin{itemize}
    \item[(i)] there is no loss of generality in maximizing the mass on the diagonal of the transport plan, as shown in Proposition~\ref{prop:coldiag};
    \item[(ii)] the sign of the imbalance cumulative function $\Xi_T(\cdot)$ along the paths of $T$ carrying positive transported mass plays a crucial role, as detailed in Proposition~\ref{prop:sign}.
\end{itemize}

We start by discussing the restrictive condition \eqref{eq:condTast}. Let us consider OT with a ground cost $d_T$  as in Section \ref{sec:tree}. If the tree $T$ with root $r$ satisfies condition \eqref{eq:condTast}, then an optimal transport plan for this OT problem is given by Proposition \ref{prop:OT_formula}. However, the condition \eqref{eq:condTast} imposes that the sign of the function $\Xi_{T}(x) = \sum_{y \lesssim x} \left( \mu(y)-\nu(y) \right)$ alternates along the paths from leaves to the root $r$ in the tree $T$. This is a restrictive condition that basically imposes the measures $\mu$ and $\nu$ to be sufficiently close leading to an optimal transport plan whose support is restricted to the pairs $\{x,x\}$, $\{x,x^+\}$ and $\{x^+,x\}$ of edges in $\ee_{T}$. It is typically violated when $\mu = \delta_{X_1}$ is the Dirac measure at $x=X_1$ and $T$ is the line tree $X_1 \to X_2 \to \cdots \to X_N = r$, since $\Xi_{T}(X_i) = 1 - \sum_{j=1}^{i} \nu(X_j) \geq 0$ for all $1 \leq i \leq N$. We also present below a simple example (beyond the case $\mu = \delta_{X_1}$)  showing that Condition \eqref{eq:condTast} is not always satisfied.

\begin{exa}
	Consider the graph $G$ given in Figure \ref{fig:graphic} where the weight function $w$ is equal to one on each edge. For the given values of $\mu$ and $\nu$ displayed in Figure \ref{fig:graphic}, we claim that there is no rooted spanning tree $T$ of the weighted graph $G$  such that Condition \eqref{eq:condTast} is satisfied. 
	\begin{figure}[htbp]
		\centering
		\includegraphics{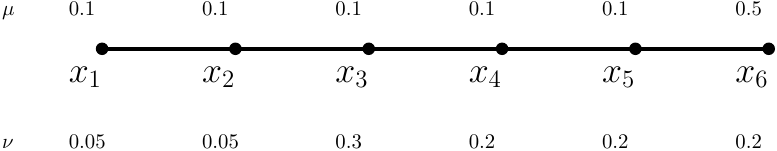}
		\caption{Values of the weights $\mu$ and $\nu$ with $N=6$. The graph $G$ is a line tree with 6 spanning trees which have the same structure with only the root placed at different vertices.}
		\label{fig:graphic}
	\end{figure}
	A short proof of this assertion is as follows. For a given rooted spanning tree $T$ of $G$, we recall that $$\Xi_T(x) = \sum_{y\lesssim x}\xi(y) = \sum_{y\lesssim x}\mu(y)-\nu(y).$$
	Notice that Condition \eqref{eq:condTast} is a relation between every node and its parent. Checking this condition translates into Table \ref{tab:example} (that gives the values of $\Xi_T$) as not having two consecutive entries in a row  with the same sign, which obviously holds on this example.
	\begin{table}[h!]
		\begin{tabular}{|l|l|l|l|l|l|l|}
			\hline
			& $x_1$ & $x_2$ & $x_3$ & $x_4$ & $x_5$ & $x_6$  \\ \hline
			$\xi = \mu-\nu$ & 0.05  & 0.05 & -0.2 & -0.1 & -0.1 & 0.3 \\ \hline
			$\Xi_T$,  for $r=x_1$ & 0 & -0.05 & -0.1 & 0.1 & 0.2 & 0.3  \\ \hline
			$\Xi_T$,  for $r=x_2$ & 0.05 & 0 & -0.1 & 0.1 & 0.2 & 0.3 \\ \hline
			$\Xi_T$,  for $r=x_3$ & 0.05 & 0.1 & 0 & 0.1 & 0.2 & 0.3 \\ \hline
			$\Xi_T$,  for $r=x_4$ & 0.05 & 0.1 & -0.1 & 0 & 0.2 & 0.3 \\ \hline
			$\Xi_T$,  for $r=x_5$ & 0.05 & 0.1 & -0.1 & -0.2  & 0 & 0.3 \\ \hline
			$\Xi_T$,  for $r=x_6$ & 0.05 & 0.1 & -0.1 & -0.2 & -0.3 & 0 \\ \hline
		\end{tabular}
		\caption{Values of the function $(\Xi_T(x))_{x \in \XX}$ for the example of graph $G$ and measures $\mu$ and $\nu$ given in  Figure \ref{fig:graphic} for each rooted spanning tree $T$ of $G$ (There is a unique spanning tree with six possible roots).} \label{tab:example}
	\end{table}
\end{exa}

In what follows, we show that the key point to build an optimal transport plan in the general case, that is without imposing Condition \eqref{eq:condTast}, is to first maximize the mass on the diagonal (see Proposition~\ref{prop:coldiag}), and then to use the imbalanced function $\Xi_T(\cdot)$ to fill its off-diagonal entries in a recursive manner as detailed in  Algorithm \ref{alg:tree_transport} below.

\subsubsection{ Maximizing the mass transported on the diagonal}\label{sec:max}

The maximization of the mass on the diagonal terms is inspired by the
Beckmann problem~\eqref{eq:beck}. Indeed, in this formulation, the behavior
at a node $x \in \XX$ is entirely determined by the sign of
$\mu(x)-\nu(x)$ through the divergence constraint imposed on admissible
flows. Intuitively, since mass can be retained at a node $x$ at zero cost,
it is natural to first remove the maximal possible amount of mass before
sending the remaining mass throughout the graph. Hence, a key result is Proposition \ref{prop:coldiag} showing that one may reduce the complexity of constructing an optimal coupling by maximizing its diagonal.

\begin{rem}
From the proof of Proposition~\ref{prop:coldiag}, it follows that every optimal transport plan $\gamma_T$ admits a  representative $\tilde{\gamma}_T$ that maximizes the mass retained at each vertex.
Proposition~\ref{prop:coldiag} also plays a key role in the implementation of Algorithm~\ref{alg:tree_transport} described in Section~\ref{sec:dynprog}, as it justifies the construction of an optimal transport plan that maximizes the mass on the diagonal.
\end{rem}

\begin{rem}
	It is easy to see that any solution $\gamma$ satisfying $\gamma(x,x) = \min\{\mu(x),\nu(x)\}$ is such that every vertex $x$ is only a sender of mass or a receiver of mass. In other words there is a row or a column on $\gamma$ which is equal to the zero vector with exception of the diagonal term.
\end{rem}

\subsubsection{Study of the sign of $\Xi_T$.}

The following proposition characterizes the evolution of the signs of the imbalance cumulative function $\Xi_{T}(\cdot)$ along paths where a positive mass is sent in a tree.

\begin{prop} \label{prop:sign}
 Let $\gamma_T$ be any optimal plan for the ground cost $d_T$ and let $(x,y) \in \XX \times \XX$ with $x \neq y$ such that $\gamma_{T}(x,y)>0$. Denote by $x_0,x_1,\ldots,x_{m}$  the  unique path  in $T$ from $x$ to $y$ (with $x_0 = x$ and $x_m=y$). Then, the sign of the function $\Xi_{T}$ satisfies:
\begin{itemize}
\item[(i)]  if $x_{i} = x_{i-1}^+$ then $ \Xi_{T}(x_{i-1}) > 0$
\item[(ii)]  if $x_{i-1} = x_{i}^+$ then $ \Xi_{T}(x_{i}) < 0$.
\end{itemize}
 \end{prop}
 \begin{proof}
Consider $(x,y) \in \XX \times \XX$  with $x \neq y$ and $\gamma_{T}(x,y)>0$. Let $x_0,x_1,\ldots,x_{m}$ be the  unique path  in $T$ from $x$ to $y$ with $x_0=x$ and $x_m = y$. Then, by definition of the distance $d_{T}$, one has that
 \begin{equation} \label{eq:dG}
d_{T}(x,y) = \sum_{i=1}^{m} w(x_{i-1},x_i) =  \sum_{i=1}^{m}  d_{T}(x_{i-1},x_{i}) .
\end{equation}
Since $u_{T}$ given by \eqref{eq:pot} is an optimal potential for the dual optimal transport with ground cost $d_{T}$, it follows by combining Proposition \ref{prop:complementary} and Proposition \ref{prop:ubar} that
 \begin{align}\label{eq:conv}
 &d_{T}(x,y) = u_{T}(x) - u_{T}(y) = \sum_{i=1}^{m} u_{T}(x_{i-1})-u_{T}(x_{i})    \\
  & = \sum_{i=1}^{m} d_{T}(x_{i-1},x_{i}) \sgn ( \Xi_{T}(x_{i-1}) ) \mathbf{1}_{\{ x_i = x_{i-1}^{+} \}} - \sum_{i=1}^{m} d_{T}(x_{i-1},x_{i}) \sgn ( \Xi_{T}(x_{i}))  \mathbf{1}_{\{ x_{i-1} = x_{i}^+ \}}  
 \end{align} 
 Now, let us denote by $x_{m_0}$ the highest vertex (first common ancestor between $x$ and $y$) in the tree $T$ along the path from $x$ to $y$, in the sense that $x_i = x_{i-1}^{+}$ for all $1 \leq i \leq m_0$ (if $m_0 \geq 1$) and $x_{i-1} = x_{i}^+$  for all $m_0 + 1 \leq i \leq m$ (if $m_0 < m$). 
 From the above equality, we thus obtain that
 \begin{equation}   \label{eq:dG2}
 d_{T}(x,y) =  \sum_{i=1}^{m_0} d_{T}(x_{i-1},x_{i}) \sgn ( \Xi_{T}(x_{i-1}) ) - \sum_{i=m_0+1}^{m} d_{T}(x_{i-1},x_{i}) \sgn ( \Xi_{T}(x_{i}))   
 \end{equation}
Therefore, each term of the sums in  \eqref{eq:dG2} is at most $d_T(x_{i-1},x_{i})$ so Equality with \eqref{eq:dG} forces $\sgn(\Xi_T(x_{i-1}))=1$ and $\sgn(\Xi_T(x_{i}))=-1,$ we conclude for the convention when $\sgn(0)=1$ that one necessarily has that $ \Xi_{T}(x_{i-1}) \geq 0$ when $x_{i} = x_{i-1}^+$ and  $ \Xi_{T}(x_{i}) < 0$ when $x_{i-1} = x_{i}^+$  for all  $1 \leq i \leq m$. To obtain that $\Xi_{T}(x_{i-1}) > 0$ is strict for $x_{i} = x_{i-1}^+$ it suffices to use the other convention in Equation \eqref{eq:conv} ($\sgn(0)=-1$), this is valid since $u_T$ is an optimal Kantorovic potential for any convention.
\end{proof}

\section{Algorithms and their properties}\label{Algos}

In this section we present two algorithms. First, we present a dynamic programming way to find an optimal transport plan of a given tree together with a proof of its correctness and properties of the solution. Secondly, we present a simulated annealing type stochastic algorithm to find a spanning tree $T_\ast$ of $G$ minimizing \eqref{min:ST}.

These algorithms should be thought as first running the second algorithm to find the structure of the spanning tree $T_\ast$ of $G$. Then, once $T_\ast$ is found it can be used as an input of the first algorithm to find an optimal plan  $\gamma_{T_\ast}$ for OT with ground cost $d_G$ as justified by Theorem \ref{thm:plan}.

\subsection{Algorithm 1: dynamic programming to find the optimal plan from a given tree $T$.}\label{sec:dynprog}

We propose below an algorithm based on the principles of dynamic programming to construct an optimal coupling  $\gamma_T$ (when the ground cost is $d_T$) that works for any pair of measures $\mu$ and $\nu$. 

The dynamic programming procedure to construct  $\gamma_T$ consists in reducing the total mass of the measures at each step of the algorithm, and iteratively fill the transport plan on a sequence of reduced problems. We start by filling the diagonal entries of the OT plan as $\gamma_T(x,x)=\min\{\mu(x),\nu(x)\}$ for each $x \in \XX$. We interpret the imbalance mass $\xi(x) = \mu(x)-\nu(x)$ as the offer when it is positive and demand otherwise. If $\xi(x) = 0$ then the node $x$ is said to be balanced. At each step, our algorithm takes a leaf of the current tree, say $x$, and it finds a specific node $y$ with the opposite demand/offer behavior (i.e. $\sgn(\xi(x)) \neq \sgn(\xi(y))$). We send from the node offering to the node demanding the maximum amount of mass possible, in such a way that one remains consistent with the total cumulative offer/demand $\Xi_T(u)$ on the nodes $u$ belonging to the unique path between $x$ and $y$. Once this is accomplished, we can update $\mu$ and $\nu$ by decreasing their total mass, which is equivalent to updating the imbalance function $\xi$ and the cumulative function $\Xi_T$. In this way, the total imbalance mass $\sum_{x \in \XX} |\xi(x)|$ decreases at each iteration of the algorithm which guarantees its convergence.

A formal description of this procedure is detailed in Algorithm \ref{alg:tree_transport}, where it is understood that $\sgn(\cdot)$ is the usual sign function with  $\sgn(0) = 0$. This algorithm proceeds by recursively peeling leaves and propagating the excess mass from offering nodes to demanding nodes, while updating the imbalance function $\xi$ and the cumulative function $\Xi_T$. 
In the literature, existing algorithms for optimal transport using tree metrics follow a similar bottom-up dynamic programming approach to produce optimal couplings \cite{Sato20,MV23}. 
The closest procedure to Algorithm \ref{alg:tree_transport} is the one described in \cite[Section 4.2]{MV23}, which (informally) operates in two distinct phases: 
in the first phase, the excess mass of all subtrees $T_x$ with $\mu(T_x) - \nu(T_x) > 0$ is pushed toward the root, while,
in the second phase, this mass is sent toward the leaves in the subtrees $T_x$ with a deficit of mass $\mu(T_x) - \nu(T_x) < 0$. 
While the algorithm in \cite[Section 4.2]{MV23} also yields an optimal transport plan, both its description and the proof of its convergence are more involved than those of Algorithm \ref{alg:tree_transport}. 
To the best of our knowledge, this construction of a transport plan has not been explicitly considered before. 
We therefore analyze below several properties of the matrix $\gamma_T$ produced by Algorithm \ref{alg:tree_transport}

\begin{algorithm}[h]
\caption{Construction of an optimal transport plan for the ground cost $d_T$}
\label{alg:tree_transport}

\textbf{Preprocessing:}\\
\Indp
$\gamma \gets 0^{\XX \times \XX}$\;
$\gamma_T(x,x) \gets \min\{\mu(x),\nu(x)\}, \quad \forall x \in \mathcal{X}$\;
$\xi \gets \mu - \nu$\; 
\While{there is a leaf $z \in T$ with $\xi(z)=0$}{
             $T \gets T \setminus \{z\}$;  
    }
\Indm
\vspace{0.5em}

\textbf{Main loop:}\\
\While{there exists a leaf $x \in T$ with $\xi(x) \neq 0$}{

Pick such a leaf $x$, and let $$m = |\xi(x)|, \; u_{-} = x, \text{ and } u = x^+;$$

\While{($\sgn(\Xi_T(u)-\Xi_T(u_{-})) = \sgn(\xi(x))$ or $\Xi_T(u)-\Xi_T(u_{-}) = 0$) and $\Xi_T(u) \neq 0$}{
\vspace{0.1cm} 
$m = \min\{m,|\Xi_T(u)|\};$ \\
$u_{-} \gets u \text{ and } u \gets u^+$;
}

    Find a vertex $y$ nearest to $u$ in the subtree $T_u$  such that:
   \begin{equation}
\sgn(\xi(x)) \neq \sgn(\xi(y)) \quad \text{and} \quad \xi(y) \neq 0, \label{condAlgo1}
\end{equation}
   \begin{equation}
\sgn(\Xi_T(v)) = - \sgn(\xi(x))  \text{ for any edge $ e = \{v^+, v \}$ on the path $P_{u,y}$ from $u$ to $y$.} \label{condAlgo2}
\end{equation}

\eIf{$y \neq u$}{
$ m = \min \{ m, m_{u,y}, |\xi(y)| \} \quad \text{where} \quad m_{u,y} = \min_{ \{v^+, v \} \in P_{u,y}} \{ |\Xi_T(v)| \} $;
}{
$ m = \min \{ m, |\xi(u)| \}$ ;
}
    \eIf{$\xi(x) > 0$}{
	$\gamma_T(x,y) \gets  \gamma_T(x,y) + m$;
        $\xi(x) \gets \xi(x) - m$;
        $\xi(y) \gets \xi(y) + m$;
    }{
	$\gamma_T(y,x) \gets  \gamma_T(y,x) + m$;
        $\xi(x) \gets \xi(x) + m$;
        $\xi(y) \gets \xi(y) - m$;
    }
    
    \While{there is a leaf $z \in T$ with $\xi(z)=0$}{
             $T \gets T \setminus \{z\}$;  
    }

    Update the cumulative function $\Xi_T$ accordingly, following the update of $\xi$ above
}
\vspace{0.5em}

\textbf{Output:} $\gamma_T$  
\end{algorithm} \vspace{0.2cm} 

\begin{prop}
\label{prop:single_write}
Algorithm~\ref{alg:tree_transport} terminates after a finite number of steps. Moreover, one has that:
\begin{itemize}
\item[(i)] during the iterations of the algorithm, every entry $\gamma_T(x,y)$ is assigned a positive value at most once. Equivalently, at a given iteration, if the algorithm transfers a mass of size $m>0$ between a leaf $x$ and a node $y$, then the ordered pair  $(x,y)$ cannot be used at a subsequent iteration to update the transport plan.
\item[(ii)]   if $\gamma_T(x,y) > 0$ with $x \neq y$ then necessarily $\gamma_T(y,x) = 0$.
\item[(iii)]  Let $P_{x,y}$ denote the set of edges on the unique path between two nodes $x$ and $y$. If $\gamma_T(x,y) > 0$, it follows that
\begin{itemize}
\item[(iii+)] $\sgn(\Xi_T(v)) = \sgn(\xi(x))$  for any $\{v,v^+\}$ belonging to $P_{x,y}$
\item[(iii-)] $\sgn(\Xi_T(v)) = -\sgn(\xi(x))$  for any $\{v^+,v\}$ belonging to $P_{x,y}$,
\end{itemize}
where $\Xi_T$ and $\xi$ are the values of the cumulative and imbalance functions at the input of the algorithm.
\end{itemize}
\end{prop}

\begin{proof}
To simplify the notation, we write $\Xi = \Xi_T$. 
Let us first prove that the Algorithm~\ref{alg:tree_transport} terminates after a finite number of steps. To this end, it suffices to show that, at each iteration of the algorithm, there always exists a node $y$ satisfying conditions \eqref{condAlgo1} and \eqref{condAlgo2}  that can be associated to $x$ with $m > 0$. Indeed, let us assume that, at a given iteration of the algorithm, one picks a leaf with $\xi(x) > 0$ (the other case being symmetric). Then, the algorithm  moves up in the tree until finding a node $u$ satisfying:
\begin{enumerate}[label=(\alph*), ref=\alph*]
\item \label{cond:a}  
$\sgn(\Xi(v)) > 0$ for all nodes $v \neq u$ on the path $P_{x,u}$ from $x$ to $u$
\item \label{cond:b}  
 $\Xi(u) - \Xi(u_{-}) < 0$ or $\Xi(u)=0$.
\end{enumerate}

Let us now consider the decomposition $\Xi(u) - \Xi(u_{-}) = \xi(u)  + \Delta(u)$ where $\Delta(u) =   \sum_{v \in \child(u) \setminus \{ u_{-} \}} \Xi(v)$, with $u_{-}$ denoting the child of $u$ that is on the path $P_{x,u}$. From Condition \eqref{cond:b}, let us consider the following cases:
\begin{enumerate}
\item[(Case 1)] if $\Xi(u) - \Xi(u_{-}) < 0$  then either $\xi(u) < 0$ or $\Delta(u) < 0$.
\begin{enumerate}
\item[(1-a)] If $\xi(u) < 0$ then the algorithm selects $y=u$ as a vertex in the subtree $T_u$ 
satisfying  \eqref{condAlgo1}, with Condition \eqref{condAlgo2} being satisfied as the path $P_{u,y}$ is empty. 
\item[(1-b)] If $\Delta(u) < 0$, there necessarily exists a node $v \in \child(u) \setminus \{ u_{-} \}$ with $\Xi(v) < 0$. Hence, thanks to the definition of the cumulative function $\Xi$, we may prove in a recursive manner that there necessarily exists a node $y$ in the subtree $T_u$ satisfying  \eqref{condAlgo1} and \eqref{condAlgo2} with $\xi(y) \neq 0$ and $\Xi(v) < 0$ on any edge $ e = \{v^+, v \}$ belonging to $P_{u,y}$.
\end{enumerate}
\item[(Case 2)] if $\Xi(u)  = 0$  then $\xi(u) + \Delta(u) = - \Xi(u_{-}) <0$ by Condition \eqref{cond:a}. Hence, either $\xi(u) < 0$ or $\Delta(u) < 0$, and we recover the previous first case above.
\end{enumerate}

Therefore, by its definition in the algorithm and after the selection of such a vertex $y$, the quantity $m$ is equal to
$$
\min\{\xi(x),-\xi(y)\}, \; \min_{(v,v^+) \in P_{x,u}} \{\Xi(v)\} \quad \text{or} \quad \min_{(v^+,v) \in P_{u,y}} \{-\Xi(v)\}.
$$
Consequently, the update $\xi(x) \gets \xi(x) - m$ and $\xi(y) \gets \xi(y) + m$ implies that either the imbalance function $\xi$ or the cumulative function $\Xi$ has one more entry equal to zero after each iteration. Therefore, during the iterations, all the entries of $\xi(x)_{x \in \XX}$ will progressively become equal to zero, implying that the algorithm terminates after a finite number of steps, which proves the first statement of Proposition \ref{prop:single_write}.

Let us now prove the  statements (i) and (ii) of Proposition \ref{prop:single_write}.  The value $\gamma_T(x,x)$ is fixed to $\min\{\mu(x),\nu(x)\}$ at initialization for all $x \in \XX$, and it is clearly not modified during the iterations of the algorithm. Then, when a leaf $x$ of the current tree $T$ with $\xi(x) \neq 0$
is chosen, a nearest node $y$ with opposite sign of $\xi$ is found, and a quantity
$
m >0
$
is transferred from the node offering to the node demanding. Assume that $\xi(x)>0$ and $\xi(y)<0$, the other case being identical by exchanging the roles of $x$ and $y$. The function $\xi$ is then updated as follows:
\[
\xi(x)\leftarrow \xi(x)-m,\qquad \xi(y)\leftarrow \xi(y)+m.
\]
If $\xi(x)-m=0$, then the node $x$ becomes balanced and it is removed from the tree $T$, and if $\xi(y)+m=0$, then the node $y$ becomes balanced and, if it is a leaf, it is removed. In either of these two cases, at least one of the vertices $x$ or $y$ becomes balanced. Since, at the following iterations, the algorithm never updates $\gamma_T(\tilde x,\tilde y)$ 
where either vertex $\tilde x$ or $\tilde y$ is balanced, it follows that the pair $(x,y)$ or $(y,x)$ cannot be used at a later iteration. In the case where $\xi(x)-m \neq 0$ and $\xi(y)+m \neq 0$, it follows from the definition of $m$ that the updating of $\xi$ results in an update of the cumulative function $\Xi$ such that there exists $v$ on the path $P_{x,y}$ from $x$ to $y$ with $\Xi(v) = 0$
with either $\{v,v^+\}$  or $\{v^+, v\}$ belonging to this path. Consequently, at subsequent iterations, the algorithm can never associate $x$ and $y$ as this can only be done along a path satisfying $\Xi(v) \neq 0$ for all such nodes $v$ thanks to Properties (\ref{cond:a}) and (\ref{cond:b}) above, which finally proves  statements (i) and (ii).

Finally, statement (iii) is an immediate consequence of  Properties (\ref{cond:a}) and (\ref{cond:b}) above which completes the proof.
\end{proof}

We shall now prove that the matrix $\gamma_T$ given by  Algorithm~\ref{alg:tree_transport} is an optimal coupling. Let us first prove that it is an admissible transport plan.

\begin{prop}
\label{prop:feasibility}
Let $\gamma_T$ be the matrix produced by Algorithm~\ref{alg:tree_transport}.
Then $\gamma_T$ is an admissible transport plan that is $\gamma_T \in \Pi(\mu,\nu)$.
\end{prop}

\begin{proof}
The algorithm starts with
\[
\gamma_T(x,x)=\min\{\mu(x),\nu(x)\},  \text{ for all } x\in \mathcal{X},
\text{ and }
\gamma_T(x,y)=0 \text{ for } x\neq y,
\]
and one has that $\xi = \mu - \nu$ at its initialization. 
At each step of the algorithm, the updating of  $\xi$ records the residual excess or deficit of mass at each vertex: a positive value $\xi(x)>0$ represents a remaining offer at $x$, and a negative value $\xi(x)<0$ a remaining demand.
For a matrix $A \in \RR^{\XX \times \XX}$, we introduce the notation 
\[
\mathrm{row}_A(x) = \sum_{y \in \XX} A(x,y),
\qquad
\mathrm{col}_A(y) = \sum_{x \in \XX} A(x,y).
\]
In what follows, we shall prove by induction, that at each step of the algorithm, the following equalities always hold:
\begin{eqnarray}
\mathrm{row}_{\gamma_T}(x)  & = &  \mu(x) - \max\{\xi(x),0\} \text{ for all } x \in \XX, \label{eq:row} \\
\mathrm{col}_{\gamma_T}(y) & = &  \nu(y) - \max\{-\xi(y),0\} \text{ for all } y \in \XX, \label{eq:col}
\end{eqnarray}
where it is understood that the value of $\xi$ is evolving according to Algorithm~\ref{alg:tree_transport}. Indeed, at initialization of the algorithm, one has that, for every $x \in \mathcal{X}$,
$$
\mathrm{row}_{\gamma_T}(x)
= \gamma_T(x,x)
= \min\{\mu(x),\nu(x)\}
= \mu(x) - \max\{\mu(x)-\nu(x),0\}
= \mu(x) - \max\{\xi(x),0\},
$$
and similarly, for every $y \in \mathcal{X}$,
$$
\mathrm{col}_{\gamma_T}(y)
= \gamma_T(y,y)
= \nu(y) - \max\{-\xi(y),0\}.
$$
Thus, Equalities \eqref{eq:row} and \eqref{eq:col} hold at initialization. Let us now consider one step of the algorithm, and suppose that it chooses a leaf $x$ and a node $y$ with opposite signs  
$\sgn(\xi(x)) \neq \sgn(\xi(y))$, and it pushes
$$
m \leq \min(|\xi(x)|,|\xi(y)|)
$$
units of mass from a node offering to a node demanding. We distinguish the following two cases:

\begin{enumerate}
\item Consider first the case $\xi(x)>0$ (offer at $x$) and $\xi(y)<0$ (demand at $y$).  
The algorithm updates the matrix $\gamma_T$ and the excess/deficit function $\xi$ as follows:
$$
\gamma_T(x,y) \leftarrow \gamma_T(x,y) + m, \qquad
\xi(x) \leftarrow \xi(x) - m, \qquad
\xi(y) \leftarrow \xi(y) + m.
$$
 Hence, assuming that Equalities \eqref{eq:row} and \eqref{eq:col} hold at the previous iteration, one has that the quantities $\mathrm{row}_{\gamma_T}(x)$ and $\mathrm{col}_{\gamma_T}(y)$ are updated as follows:
\[
\mathrm{row}_{\gamma_T}(x)
\leftarrow \mathrm{row}_{\gamma_T}(x) + m
= \big(\mu(x) - \xi(x)\big) + m
= \mu(x) - (\xi(x)-m)
= \mu(x) - \max\{\xi(x)-m,0\},
\]
since $\xi(x)-m\ge 0$ by definition of $m$.
Similarly,
\[
\mathrm{col}_{\gamma_T}(y)
\leftarrow \mathrm{col}_{\gamma_T}(y) + m
= \big(\nu(y) + \xi(y)\big) + m
= \nu(y) - (-\xi(y)-m)
= \nu(y) - \max\{-(\xi(y)+m),0\},
\]
since $\xi(y)+m \leq 0$  by definition of $m$.
All other rows and columns of $\gamma_T$ remain unchanged, and therefore  Equalities \eqref{eq:row} and \eqref{eq:col} still hold after an iteration of the algorithm.

\item The second case $\xi(x)<0$ and $\xi(y)>0$ is symmetric,  and the result follows by exchanging the roles of $\xi(x)$ and $\xi(y)$.
\end{enumerate}

At each iteration of the algorithm, a non-negative amount $m$ is added to one entry of $\gamma_T$,   and therefore, one always has that $\gamma_T(x,y) \geq 0$ for all $(x,y) \in \XX \times \XX$ along all the iterations. 

Thanks to Proposition \ref{prop:single_write}, we notice that the algorithm terminates after a finite number of steps and, at termination one necessarily has that $\xi(x)=0$ for all $x$. Therefore, at the last iteration, using  Equalities \eqref{eq:row} and \eqref{eq:col}, it follows that
$$
\mathrm{row}_{\gamma_T}(x) = \mu(x) - \max\{0,0\} = \mu(x), \qquad
\mathrm{col}_{\gamma_T}(y) = \nu(y) - \max\{0,0\} = \nu(y).
$$
Therefore the matrix $\gamma_T$ produced by Algorithm~\ref{alg:tree_transport} has non-negative entries and its marginals (when summing either over rows or columuns) are $\mu$ and $\nu$. It is thus an admissible transport plan which completes the proof.
\end{proof}

The following proposition proves the optimality of $\gamma_T$, and it also shows how the mass is transported along paths in the tree $T$. To this end, we introduce the following notation: for a given edge $e = \{u,u^+\}$, we denote by $T_u$ and $T_u^c$ the decomposition of $T$ into two subtrees  that are connected by this edge.
 
\begin{prop} \label{prop:optim_plan} 
The matrix $\gamma_T$ produced by
Algorithm~\ref{alg:tree_transport} is an optimal transport plan. Moreover, for a given edge $e = \{u,u^+\}$, if $\Xi_T(u) = \mu(T_u) - \nu(T_u) > 0$ then
\begin{equation}
\sum_{x\in T_u}\sum_{y\in T_u^c} \gamma_T(x,y) = \mu(T_u) - \nu(T_u), \text{ and  } \sum_{x\in T_u^c}\sum_{y\in T_u} \gamma_T(x,y) = 0, \label{eq:pos}
\end{equation}
whereas if  $\Xi_T(u) = \mu(T_u) - \nu(T_u) < 0$ then
\begin{equation}
\sum_{x\in T_u}\sum_{y\in T_u^c} \gamma_T(x,y) = 0, \text{ and  } \sum_{x\in T_u^c}\sum_{y\in T_u} \gamma_T(x,y) = \nu(T_u) - \mu(T_u). \label{eq:neg}
\end{equation}
\end{prop}

\begin{proof}
Thanks to Proposition \ref{prop:feasibility}, $\gamma_T$ is an admissible transport plan. To prove its optimality, it remains to show that its total cost equals the value of  the $K$-distance $K_{d_T}(\mu,\nu)$. To this end, recall  that for $(x,y)\in \XX \times \XX$,  $P_{x,y}$ denotes
the  set of edges on the unique path from $x$ to $y$ in the tree $T$, and  that
$
d_T(x,y) \;=\; \sum_{e \in P_{x,y}} w(e)
$
with $w(e) = w(u,v)$ for $e =\{u,v\}$.
Therefore, the cost of the transport plan $\gamma_T$ may be decomposed as follows:
\begin{eqnarray}
\sum_{x,y \in \XX} \gamma_T(x,y)\, d_T(x,y) & =&  \sum_{x,y \in \XX} \gamma_T(x,y) \sum_{e \in P_{x,y}} w(e) \nonumber \\
 & = & \sum_{e \in \ee_T} w(e) \sum_{x,y \in \XX} \gamma_T(x,y)\,\mathbf{1}_{\{e \in P_{x,y} \}}, \label{eq:keydecomp}
\end{eqnarray}
where $\mathbf{1}_{\{\cdot\}}$ is the indicator function.
 Then, for a given edge $e = \{u,u^+\}$, the tree $T$ may be decomposed into two subtrees $T_u$ and $T_u^c$ that are connected by this edge. Now, the fact that $\mathbf{1}_{\{e \in P_{x,y} \}} = 1$ is equivalent to the property that either $(x \in T_u, y \in T_u^c)$ or  $(x \in T_u^c, y \in T_u)$. Therefore, having fixed the edge $e = \{u,u^+\}$, one has that
\begin{equation}
 \sum_{x,y \in \XX} \gamma_T(x,y)\,\mathbf{1}_{\{e \in P_{x,y} \}}
  = \sum_{x\in T_u}\sum_{y\in T_u^c} \gamma_T(x,y)
    + \sum_{x\in T_u^c}\sum_{y\in T_u} \gamma_T(x,y). \label{eq:decomp}
\end{equation} 
The right-hand side of \eqref{eq:decomp} is interpreted as the  total amount of mass that crosses the edge $e$ in both directions. It remains to show that, for the plan $\gamma_T$ produced by Algorithm~\ref{alg:tree_transport},
and for each edge $e=(u,u^+)$, one of the two sums in the right-hand side of \eqref{eq:decomp} vanishes, while the other one
equals $|\mu(T_u)-\nu(T_u)|$, meaning that the total mass crossing the edge $e$ goes in the direction of the sign of $\mu(T_u)-\nu(T_u)$.

To prove this property, we let $\Xi(u) =  \sum_{x \lesssim u} \xi(x) =  \mu(T_u) - \nu(T_u)$ be the imbalance mass of subtree $T_u$, and suppose that $\mu(T_u) - \nu(T_u) > 0$ at the initialization of the algorithm. During the iterations of the algorithm, a leaf $x \in T_u$ with $\xi(x) > 0$ will eventually be  picked and associated to a node $y$ with $\xi(y) < 0$. Then,  for a positive $m > 0$, the algorithm will update the transport plan and the function $\xi$ as follows: $\gamma_T(x,y) \gets  m$, $\xi(x) \gets \xi(x) - m$ and $\xi(y) \gets \xi(y) + m$. Therefore, if $y \in T_u$ then the value of the cumulative function $\Xi(u)$ will not change. To the contrary, if $y \in T_u^c$ then $\Xi(u)$ will decrease by an amount of $m$, meaning that a quantity $m$ of mass is pushed from $T_u$ to $T_u^c$.  As shown in the proof of Proposition \ref{prop:single_write},  the value of the cumulative function $\Xi(u)$ decreases along the iterations until reaching zero, implying that the total amount of mass transferred from  $T_u$ to $T_u^c$ through the edge $e=(u,u^+)$ is
$$
\mu(T_u) - \nu(T_u) = \sum_{x\in T_u}\sum_{y\in T_u^c} \gamma_T(x,y),
$$
and that
\begin{equation}
\sum_{x\in T_u^c}\sum_{y\in T_u} \gamma_T(x,y) = 0. \label{eq:sum_null}
\end{equation}
Indeed, by construction of the algorithm, sending mass from a leaf $x \in T_u^c$ with $\xi(x) > 0$ to a demanding node $y \in T_u$ with $\xi(y) < 0$  by passing through the edge $e=(u,u^+)$ necessarily requires that $\Xi(u) < 0$ which is not possible as the sign of $\Xi(u)$ always remains unchanged (and thus positive here) during the iterations of the algorithm by the construction of the quantity of mass $m$ that can be transferred between two nodes. Hence, transferring mass from $T_u^c$ to  $T_u$ is not allowed implying that $\gamma_T(x,y) = 0$ for any pair $(x,y)$ with $x\in T_u^c$ and $y\in T_u$ which proves that equality \eqref{eq:sum_null} holds. This completes the proof.
\end{proof}

Equations \eqref{eq:pos} and  \eqref{eq:neg} may be interpreted as follows. Recall that, for every $(x,y)\in \XX \times \XX$, we denote by $P_{x,y}$ the unique path going from $x$ to $y$ in the tree $T$. Then, Proposition \ref{prop:optim_plan} implies that if $\gamma_T(x,y)>0$ for some $(x,y)\in \XX \times \XX$, then $\gamma_T(u,v)=0$ for every $u,v\in \XX$ such that the oriented path $P_{u,v}$ in the tree $T$ uses an edge in $P_{x,y}$ in the opposite direction. We illustrate this property in Figure \ref{fig_lem}.

\begin{rem}
From the relationship  \eqref{eq:flow_gamma} between a transport plan and a flow on a tree $T$, the output of Algorithm~\ref{alg:tree_transport} and the result of Proposition \ref{prop:optim_plan} yield the flow
$$
f_T(u,u^+) =  \sum_{(x,y) : (u,u^+) \in P_{x,y}}\gamma_T(x,y) = \sum_{x\in T_u}\sum_{y\in T_u^c} \gamma_T(x,y) = \Xi_T^+(u)
$$
and
$$
f_T(u^+,u) =  \sum_{(x,y) : (u^+,u) \in P_{x,y}}\gamma_T(x,y) = \sum_{x\in T_u^c}\sum_{y\in T_u} \gamma_T(x,y) = \Xi_T^-(u)
$$
which is an optimal flow for the Beckmann problem by Proposition \ref{prop:flow_formula}. Conversely, Algorithm~\ref{alg:tree_transport}  can also be interpreted as the construction of an optimal coupling from the knowledge of the  optimal flow $f_T$ as it only depends on the imbalance cumulative function $\Xi_T$. 
\end{rem}

   \begin{figure}[htbp]
       \centering
       \includegraphics[scale=0.75]{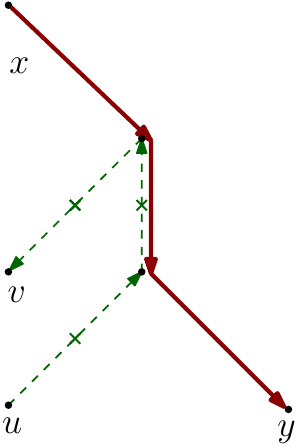}
       \caption{Illustration of Equations \eqref{eq:pos} and  \eqref{eq:neg} from Proposition \ref{prop:optim_plan}, which says that the transit of mass can be made only in one direction. The red path represents a positive transport $\gamma_T(x,y) > 0$ and the dashed path represent a transport that is not allowed since $\gamma_T(u,v)=0$.}
       \label{fig_lem}
   \end{figure}
   
 Finally, using the property that $\gamma_T(x,x)=\min\{\mu(x),\nu(x)\}$ and the fact that $\gamma_T$ is an admissible transport plan (that is $\sum_{x \in \XX } \gamma_T(x,y)= \nu(y)$ and  $\sum_{y \in \XX } \gamma_T(x,y)= \mu(x)$), we obtain the following property.
 
\begin{lem}\label{lem:diag_support}
Let $\gamma_T$ be the optimal transport plan produced by
Algorithm~\ref{alg:tree_transport}.
Then, for any $x,y \in \XX$:
\begin{itemize}
    \item[(i)] If $\mu(x) \le \nu(x)$, then $\gamma_T(x,\tilde{y}) = 0$ for all $\tilde{y} \neq x$.
    \item[(ii)] If $\nu(y) \le \mu(y)$, then $\gamma_T(\tilde{x},y) = 0$ for all $\tilde{x} \neq y$.
\end{itemize}
\end{lem}

\subsection{Algorithm 2: simulated annealing to find an optimal spanning tree of $G$ } \label{sec:algoSA}

We now describe a stochastic algorithm based on simulated annealing (SA) to find a solution to the optimisation problem \eqref{min:ST} over the set $\ST$ of spanning trees of a given edge-weighted graph $G = (\XX,\ee_G,w)$. The algorithm starts with an initial spanning tree of $G$ that is denoted as $T_0 = (\XX, \ee_0, w)$ and  rooted at an arbitrary vertex $r_0$.  Then, our stochastic algorithm at iteration $m \geq 1$ consists in randomly generating a new candidate spanning tree $\hat{T}_m$ rooted at a new root $\hat{r}_m$. Based on the principles of simulated annealing, we first compute the following Hamiltonian function
$$
H(T_{m-1},\hat{T}_m) = K_{d_{T_{m-1}}}(\mu,\nu) - K_{d_{\hat{T}_m}}(\mu,\nu).%
$$
We define the probability
$$
p_m = \min(1,\exp(\beta_m H(T_{m-1},\hat{T}_m))),
$$
where $\beta_m > 0$ is a temperature parameter that is calibrated during the iterations using an adaptive cooling procedure, to be described below. We accept or reject the candidate spanning tree from a sample $U_m$ of the uniform distribution on $[0,1]$ and $\hat{T}_m$ as follows:
\begin{itemize}
	\item[-] if  $U_m \leq p_m$ then $T_m =  \hat{T}_m$ and $r_m = \hat{r}_m$ (accept),
	\item[-] if  $U_m > p_m$ then $T_m =  T_{m-1}$ and $r_m = r_{m-1}$  (reject).
\end{itemize}

Notice that $H(T_{m-1},\hat{T}_m)$ is positive if the candidate tree $\hat{T}_m$ has smaller AE-norm than $T_{m-1}$, so we improve the value of the functional of the optimisation problem \eqref{min:ST} by replacing $T_{m-1}$ with  $\hat{T}_m$. When this happens, the parameter $p_m = 1$ and the candidate spanning tree is thus always accepted.

\subsubsection{Choice of the random transitions and fast evaluation of the Hamiltonian}
The crucial steps in the above algorithm are the random choice of a candidate spanning tree $\hat{T}_m$ and the fast evaluation of the Hamiltonian function. These two steps are described below. \\

{\it Random choice of $\hat{T}_m$}. From iteration $m-1$, one has a weighted spanning tree $T_{m-1} = (\XX,\ee_{m-1},w)$ rooted at $r_{m-1}$, we can orient it from the leaves to the root due to the partial ordering that has been chosen in Section \ref{sec:tree}. Hence, every vertex but the root has exactly one outgoing edge in this orientation.
Starting from the  spanning tree $T_{m-1}$ rooted at $r_{m-1}$, we  randomly choose a vertex $\hat{r}_m \neq r_{m-1}$ by a uniform sampling in the neighbors  of  $r_{m-1}$ in the graph $G = (\XX,\ee_G,w)$. Then, we define $\hat{T}_m = (\XX,\hat{\ee}_{m},w)$ as the tree rooted at $\hat{r}_m$, formed by the new set of edges
$$
\hat{\ee}_{m} =  (\ee_{m-1}  \setminus{ \{\underline{e}_m \} }) \cup\{{\overline{e}_m} \},
$$
for
$$
\overline{e}_m=\{r_{m-1},\hat{r}_m\} , \mbox{ and } \underline{e}_m = \{\hat{r}_m,\hat{r}_{m}^{+}\};
$$
where the edge $(\hat{r}_m,\hat{r}_{m}^{+})$, a directed version of $\underline{e}_m$, was the unique outgoing edge of $\hat{r}_m$ in $T_{m-1}$. It is easy to verify that
$$
\hat{T}_m:= (\XX,\hat{\ee}_{m},w)
$$
is still a spanning tree, since it is still connected and every vertex but the root has exactly one outgoing edge.\\

{\it Fast computation of the Hamiltonian.} In the graph $G$, we denote by $\mathcal{C}_m$ the cycle obtained by the union of the path from  $\hat{r}_m$ to $r_{m-1}$ in the tree $T_{m-1}$ and the new edge $\overline{e}_m =  \{r_{m-1},\hat{r}_m\}$.  Then, a key remark is that $\Xi_{T_{m-1}}(x) = \Xi_{\hat{T}_m}(x)$ for all edges $\{x,x^+\}$  not belonging to this cycle. Therefore, noticing that
$$
\Xi_{\hat{T}_m}(x) = \Xi_{T_{m-1}}(x) - \Xi_{T_{m-1}}(\hat{r}_m) \text{ for } \{x,x^+\} \in \mathcal{C}_m \setminus \{  \overline{e}_m,\underline{e}_m \},
$$
and that
$$
\Xi_{\hat{T}_m}(r_{m-1}) = - \Xi_{T_{m-1}}(\hat{r}_m),
$$
it follows that computing the Hamiltonian simply amounts to a summation over edges only belonging to the cycle $\mathcal{C}_m$ as shown by the expression below:
\begin{eqnarray*}
	H(T_{m-1},\hat{T}_m) & =  & \sum_{x \in \XX \setminus \{r_{m-1}\}} w(x,x^{+}) |\Xi_{T_{m-1}}(x)| - \sum_{x \in \XX \setminus \{\hat{r}_m \}} w(x,x^{+}) |\Xi_{\hat{T}_m}(x)|  \\
	& = &  \sum_{ \{x,x^+\} \in \mathcal{C}_m \setminus{\overline{e}_m, \underline{e}_m} } w(x,x^{+}) |\Xi_{T_{m-1}}(x)| - \sum_{\{x,x^+\} \in \mathcal{C}_m  \setminus{\overline{e}_m, \underline{e}_m}  } w(x,x^{+}) |\Xi_{\hat{T}_m}(x)| \\
	& & + w(\hat{r}_m,\hat{r}_{m}^{+}) |\Xi_{T_{m-1}}(\hat{r}_m)|  - w(r_{m-1},\hat{r}_m)  |\Xi_{\hat{T}_m}(r_{m-1})| \\
	& = &  \sum_{ \{x,x^+\} \in \mathcal{C}_m \setminus{\overline{e}_m, \underline{e}_m} } w(x,x^{+}) \left( |\Xi_{T_{m-1}}(x)| - |  \Xi_{T_{m-1}}(x) - \Xi_{T_{m-1}}(\hat{r}_m)|\right) \\
	& & +  \left( w(\hat{r}_m,\hat{r}_{m}^{+}) - w(r_{m-1},\hat{r}_m) \right) |\Xi_{T_{m-1}}(\hat{r}_m)|,
\end{eqnarray*}
which is an expression that is fast to compute for the graph considered in these numerical experiments. 

\subsubsection{An adaptive cooling schedule}

Finally, we use an adaptive strategy inspired by the general adaptive MCMC
framework of \cite{AndrieuThoms2008}, in which the temperature parameter
$\beta_m$ is progressively adjusted according to the observed acceptance rate
of the proposed transitions from $T_{m-1}$ to $\hat{T}_m$. The underlying idea is to control the frequency with which proposed transitions are accepted.
When proposed transitions are accepted too frequently, the temperature is increased
in order to allow for a broader exploration of the space of trees.
Conversely, when the acceptance rate becomes too small, the temperature is
decreased in order to concentrate the search around the current tree. More precisely, let $a_m$ denote the empirical acceptance rate at iteration $m$,
computed over a sliding window of the last $M$ iterations (with $M=100$ in our
experiments), and let $a^\ast \in (0,1)$ be a prescribed target acceptance rate.
Starting from an initial temperature $\beta_0 > 0$, the temperature is updated
according to the multiplicative rule
$$
\beta_{m+1}
=
\beta_m \bigl(1 + \eta (a_m - a^\ast)\bigr),
$$
where $\eta > 0$ is a  parameter chosen sufficiently small to ensure a smooth
evolution of the temperature across the iterations. When $a_m > a^\ast$, the temperature $\beta_m$ increases, which results in a
higher probability of accepting moves that increase the Hamiltonian $H(T_{m-1},\hat{T}_m)$.
On the contrary, when $a_m < a^\ast$, the temperature decreases, thereby
reducing the acceptance of such transitions, which  favors exploration
around the current tree. The adaptation parameter $\eta$ is chosen sufficiently small to ensure a smooth
evolution of the temperature across iterations. In all numerical experiments, we use $\beta_0 = 0.1$, $a^\ast = 0.01$, and
$\eta = 0.01$.

\section{Numerical experiments} \label{sec:num}

Stochastic algorithms for computational OT~\cite{Stochastic_Bigot_Bercu,NIPS2016_6566} 
are designed to solve the Kantorovich dual formulation of optimal transport. 
In this paper, we propose a stochastic algorithm based on simulated annealing 
to solve the minimization problem~\eqref{min:ST} over the space $\ST$ of rooted 
spanning trees of $G$, which is a finite but typically very large set. 
This numerical procedure returns an optimal spanning tree $T_\ast$, which is then 
used to construct an optimal transport plan via Algorithm~\ref{alg:tree_transport}, 
as well as a dual potential $u_G$ given by~\eqref{eq:pot_intro}. 
To evaluate the performance of our stochastic algorithm, we compare its outputs 
with those produced by the Python Optimal Transport (POT) toolbox~\cite{flamary2024pot,flamary2021pot}, 
which computes a transport plan and a dual potential using the network simplex algorithm from the knowledge of the measure $\mu$ and $\nu$ and a cost matrix $C = \left\{c(x,y)\right\}_{(x,y) \in \XX \times \XX}$. 

In the numerical experiments, we consider two choices of cost matrix, namely $c(x,y) = d_G(x,y)$ and $c(x,y) = d_{T_\ast}(x,y)$, corresponding to different ground costs. Since these ground costs are distances induced by a graph, the OT problem can also be formulated as an instance of a minimum-cost flow problem via the Beckman formulation \eqref{eq:beck}. Nevertheless, for simplicity in these illustrative numerical experiments, we restrict our comparison to the network simplex algorithm implemented in the POT library, which readily provides both an optimal transport plan and a dual potential.

\subsection{Description of the finite metric space $\XX$ and graph $G$}

We report numerical experiments on OT between images of size $p \times p$ taken from the MNIST database \cite{MNIST}. The metric space $\XX$ that we consider consists of $N$ distinct points $X_1,\dots,X_N$ representing the pixel locations of a two-dimensional square image. The underlying graph is thus the regular two-dimensional lattice 
of size $p \times p$, whose vertices can be identified with the subset
$$
\XX = \{(k,\ell) \in \mathbb{N}^2 : 1 \le k,\ell \le p \},
$$
so that the cardinality of $\XX$ is $N = p^2$. Each vertex of $\XX$ corresponds to one pixel of a given image. The graph $G$ with vertex set $\mathcal{X}$ is defined as the Cartesian grid graph: two vertices $(k,\ell)$ and $(k',\ell')$ are adjacent if and 
only if they differ by one unit in exactly one coordinate, that is, for a given vertex $(k,\ell)$ its four neighbours are 
$$
(k',\ell') \in \{ (k+1,\ell),\ (k-1,\ell),\ (k,\ell+1),\ (k,\ell-1) \},
$$
with an appropriate modifications for pixels at the boundaries of the images that have either three neighbours or two for pixels at the corners.  All resulting edges are assigned the weight $1/N$.
Examples of two such measures $\mu$ and $\nu$ supported on the underlying graph $G$ as described above are displayed in Figure \ref{fig:mu0nu1_both_small} and Figure \ref{fig:mu0nu1_both_large}. Raw digits images from the MNIST database have been subsampled by either a factor 4 or a factor 2 for better visualisation of the graph.  We consider two settings. In the first one, the measures $\mu$ and $\nu$ are obtained by normalizing the pixel values of two raw images so that their total mass equals one. In the second setting, the images are perturbed by additive random noise before normalization. In this noisy case, the resulting measures $\mu$ and $\nu$ can be regarded as weakly non-degenerate in the sense of Definition~\ref{def:weak-non-deg}.


\begin{figure}[htbp]
	\centering
	{\subfigure[Measure $\mu$ (noiseless setting) ]{\includegraphics[width=0.45 \textwidth,height=0.45\textwidth]{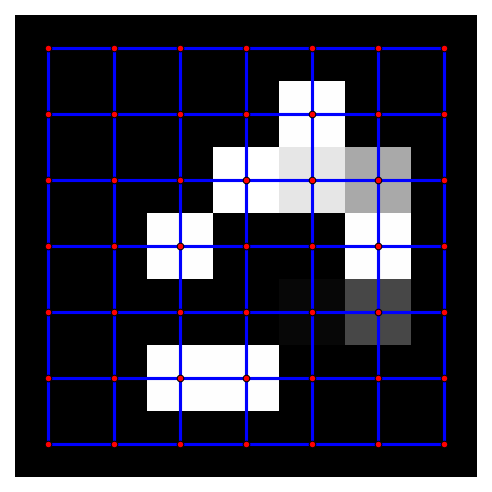}}}
	{\subfigure[Measure $\nu$  (noiseless setting) ]{\includegraphics[width=0.45 \textwidth,height=0.45\textwidth]{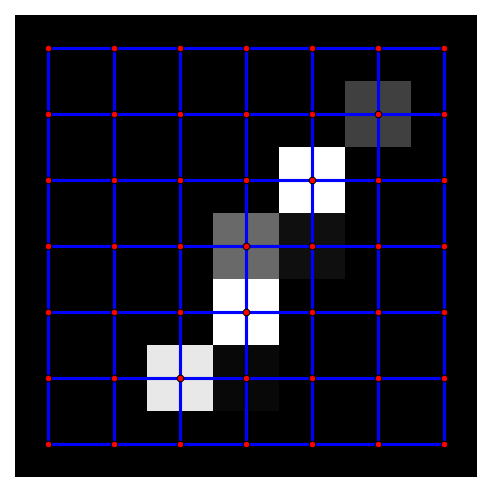}}}
	
	{\subfigure[Measure $\mu$ (noisy setting) ]{\includegraphics[width=0.45 \textwidth,height=0.45\textwidth]{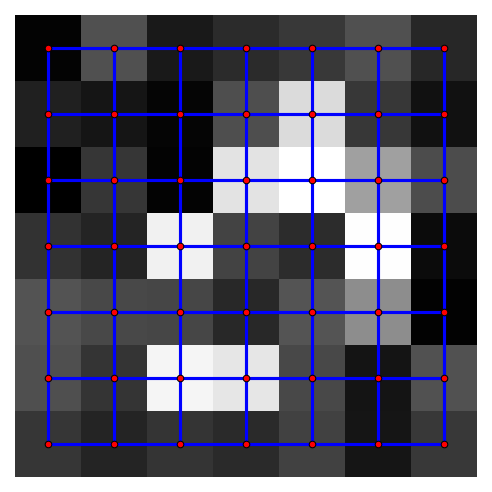}}}
	{\subfigure[Measure $\nu$  (noisy setting) ]{\includegraphics[width=0.45 \textwidth,height=0.45\textwidth]{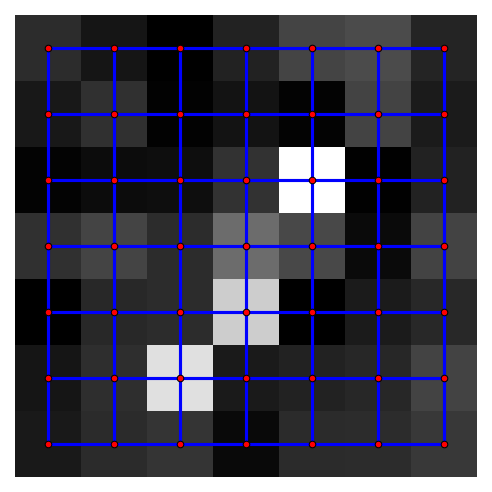}}}
	
	\caption{Probability measures $\mu$ and $\nu$ supported on a regular two-dimensional lattice $G$ of size $p \times p$, with $p = 7$, in both the noiseless and noisy settings. In all the figures, the red dots represent the set $\XX$ of vertices. The edges  of the graph $G$ are shown as  vertical and horizontal blue segments. }  \label{fig:mu0nu1_both_small}
\end{figure}

\begin{figure}[htbp]
	\centering
	{\subfigure[Measure $\mu$ (noiseless setting) ]{\includegraphics[width=0.45 \textwidth,height=0.45\textwidth]{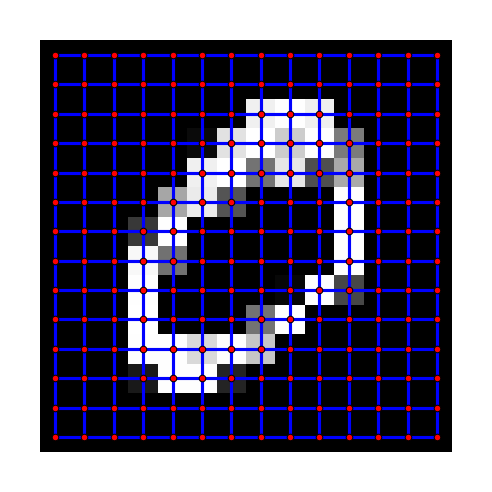}}}
	{\subfigure[Measure $\nu$  (noiseless setting) ]{\includegraphics[width=0.45 \textwidth,height=0.45\textwidth]{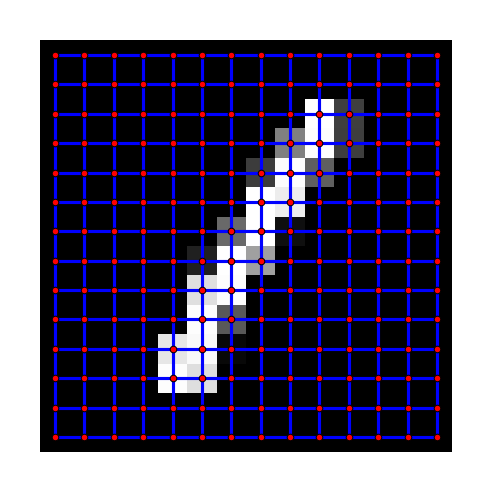}}}
	
	{\subfigure[Measure $\mu$ (noisy setting) ]{\includegraphics[width=0.45 \textwidth,height=0.45\textwidth]{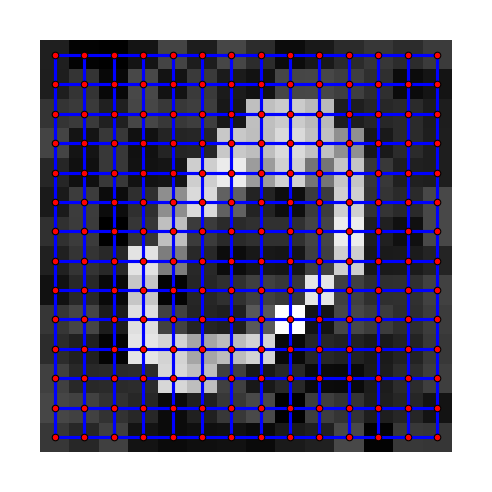}}}
	{\subfigure[Measure $\nu$  (noisy setting) ]{\includegraphics[width=0.45 \textwidth,height=0.45\textwidth]{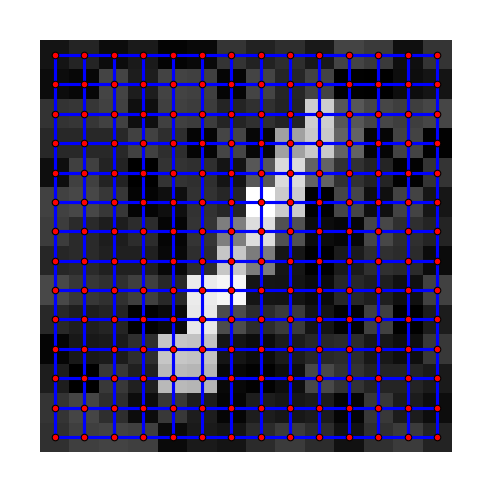}}}
	
	\caption{Probability measures $\mu$ and $\nu$ supported on a regular two-dimensional lattice $G$ of size $p \times p$, with $p = 14$, in both the noiseless and noisy settings. In all the figures, the red dots represent the set $\XX$ of vertices. The edges  of the graph $G$ are shown as  vertical and horizontal blue segments. }  \label{fig:mu0nu1_both_large}
\end{figure}

\subsection{Evaluation of the performances of the SA algorithm}

For the probability measures displayed in Figure \ref{fig:mu0nu1_both_small} and Figure \ref{fig:mu0nu1_both_large}, we report results on the computation of an optimal coupling using Algorithm  \ref{alg:tree_transport} and the dual potential given by \eqref{eq:pot_intro} thanks to the use of the SA algorithm described in Section \ref{sec:algoSA} to build an optimal spanning tree $T_{\ast}$ solving the minimization problem~\eqref{min:ST}. The performances of our approach are evaluated by comparing its outputs with those given by the network simplex algorithm implemented in the POT library  that yields the ground truth value of the K-distance, an optimal coupling and a Kantorovich dual potential. The results are reported in Figure \ref{fig:mu0nu1_noiseless_small} to Figure \ref{fig:mu0nu1_noise_large}, where the initial tree $T_0$ and the optimal spanning tree $T_{\ast}$ found by SA are displayed. In all the figures displaying trees, the red dots represent the set $\mathcal{X}$ of vertices, while the edges of a tree are shown as horizontal and vertical blue segments, and the root vertex is highlighted in green.

In all setting (noiseless/noisy and $p=7$/$p=14$), the SA algorithm converges in the sense that $ K_{d_{T_{m}}}(\mu,\nu) $ equals  $ K_{d_{G}}(\mu,\nu) $ (computed with POT) after $m=10^6$ (when $p=7$) or $m=10^8$ (when $p=14$) iterations. In the noiseless setting, the measures $\mu$ and $\nu$ being degenerated, the dual potential  given by \eqref{eq:pot_intro} differs from the one computed with POT, as it can be seen from Figures \ref{fig:mu0nu1_noiseless_small}(c,d) and Figures \ref{fig:mu0nu1_noiseless_large}(c,d). Conversely, in the noisy setting, the  measures $\mu$ and $\nu$ are weakly non-degenerated (almost surely), and there exists a unique dual potential (almost surely) that one may construct using either Algorithm  \ref{alg:tree_transport} or POT, as it can be seen from Figures \ref{fig:mu0nu1_noise_small}(c,d) and Figures \ref{fig:mu0nu1_noise_large}(c,d).

Finally, in Figures \ref{fig:mu0nu1_noiseless_small}(f,g,h) to Figure \ref{fig:mu0nu1_noise_large}(f,g,h),  optimal transport plans are visualized using green arrows whose thickness is proportional to the amount of mass transported between pairs of vertices. These arrows are displayed together with the superimposed optimal spanning tree $T_\ast$. We compare the 
optimal coupling $\gamma_{T_\ast}$ found using Algorithm  \ref{alg:tree_transport} with the optimal coupling found using POT  for the ground cost $d_{T_\ast}$, and the optimal coupling found using POT  for the ground cost $d_{G}$. For each value of $p$ and noiseless/noisy setting, we observe that these three transport plans are always different.

Moreover, it can be seen from  Figures \ref{fig:mu0nu1_noiseless_small}(f,g) to Figure \ref{fig:mu0nu1_noise_large}(f,g) that the plans computed using either Algorithm  \ref{alg:tree_transport} or OT with ground cost $d_{T_\ast}$ satisfies the  property in Theorem \ref{thm:plan} claiming that  for every $(x,y) \in \XX \times \XX$ with $\gamma_{T_\ast}(x,y)>0$, one has that $d_G(x,y) = d_{T_\ast}(x,y)$, meaning that  the edges with positive transport in $T_\ast$ belong to the geodesic graph of $G$. 

\begin{figure}[htbp]
	\centering
	{\subfigure[Initial spanning tree $T_0$ to start SA]{\includegraphics[width=0.45 \textwidth,height=0.45\textwidth]{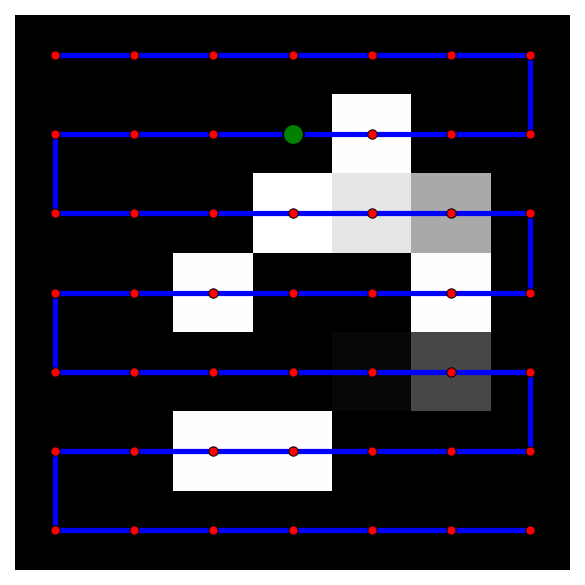}}}
	{\subfigure[Optimal spanning tree $T_\ast$ found by SA  ]{\includegraphics[width=0.45 \textwidth,height=0.45\textwidth]{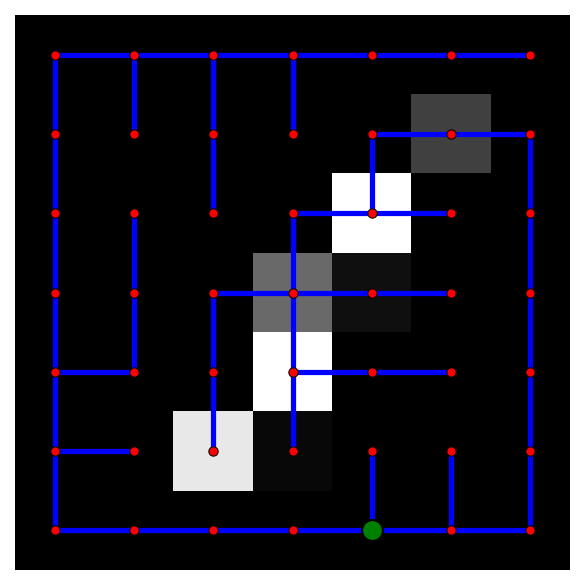}}}
	
	{\subfigure[Dual potential $u_{T_\ast}$ found using SA and formula \eqref{eq:pot_intro}]{\includegraphics[width=0.25\textwidth,height=0.25\textwidth]{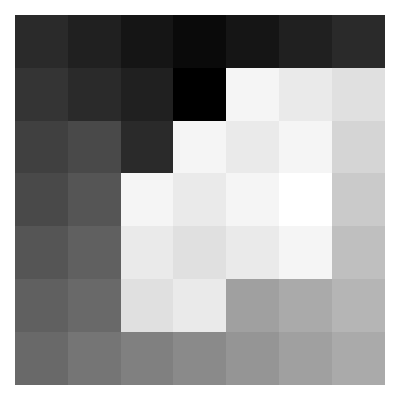}}}
	{\subfigure[Dual potential $u_G$ found using POT and the ground cost $d_{G}$]{\includegraphics[width=0.25\textwidth,height=0.25\textwidth]{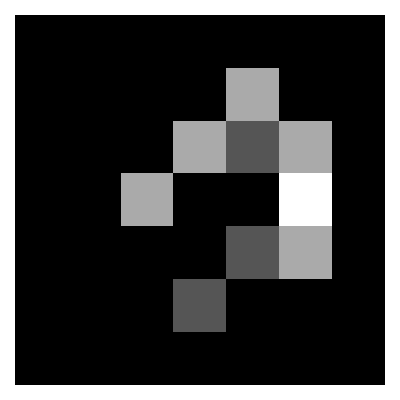}}}
	{\subfigure[Convergence of SA along $10^6$ iterations]{\includegraphics[width=0.4 \textwidth,height=0.25\textwidth]{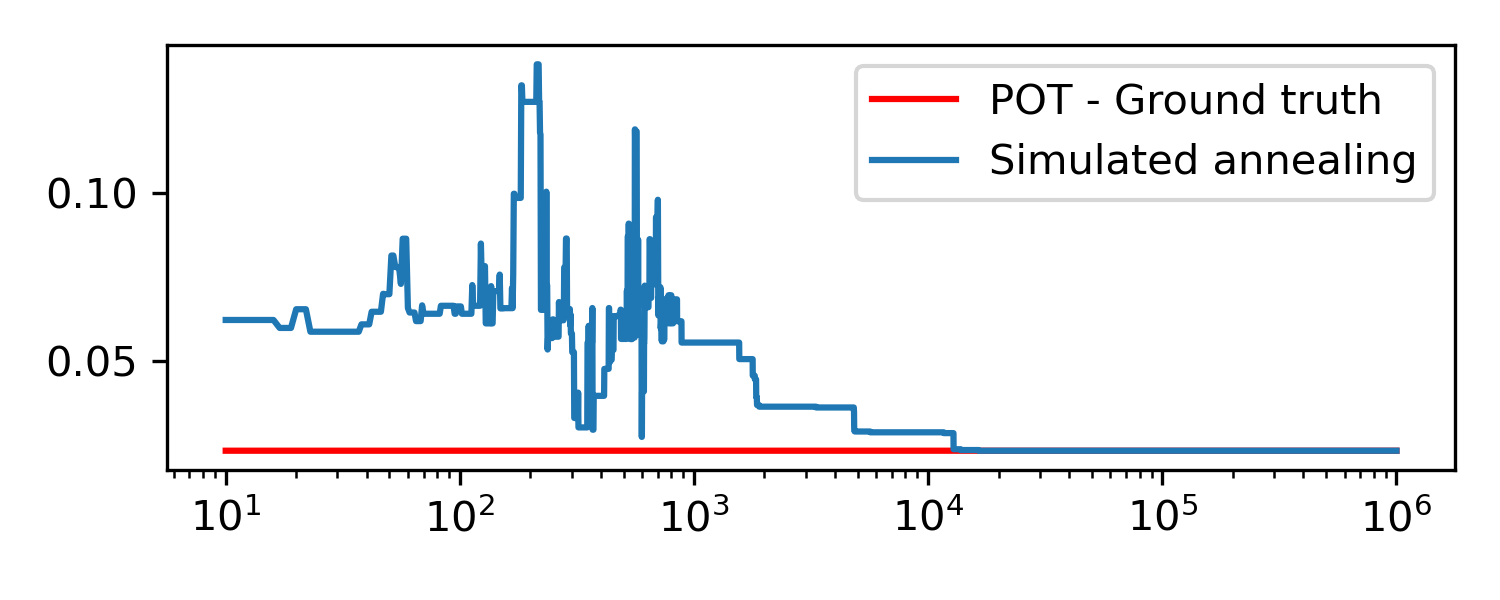}}}
	
	{\subfigure[Optimal coupling $\gamma_{T_\ast}$ found using Algorithm  \ref{alg:tree_transport}]{\includegraphics[width=0.32 \textwidth,height=0.32\textwidth]{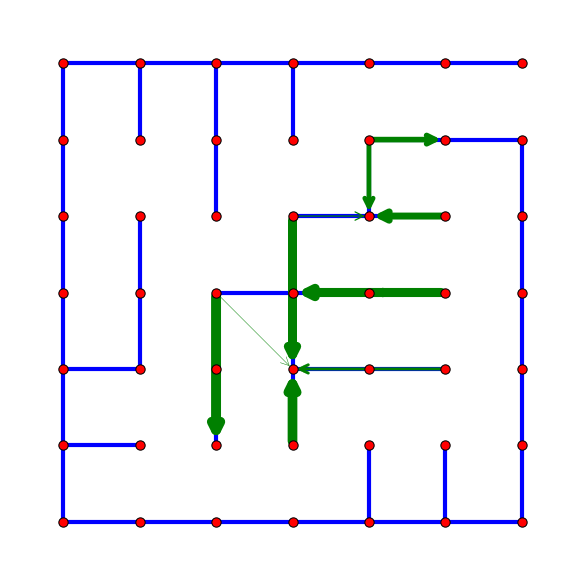}}}
	{\subfigure[Optimal coupling found using POT  and the ground cost $d_{T_\ast}$]{\includegraphics[width=0.32 \textwidth,height=0.32\textwidth]{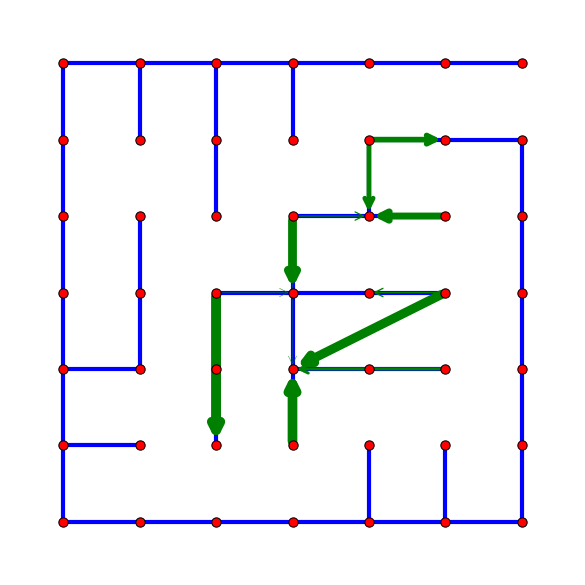}}}
	{\subfigure[Optimal coupling found using POT  and the ground cost $d_{G}$]{\includegraphics[width=0.32 \textwidth,height=0.32\textwidth]{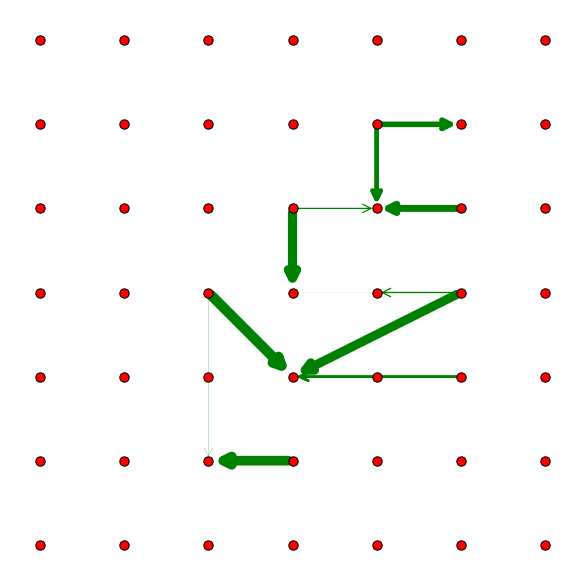}}}
	
	\caption{OT using Simulated Annealing (SA) in the noiseless setting  $p=7$.
	}   \label{fig:mu0nu1_noiseless_small}
\end{figure}

\begin{figure}[htbp]
	\centering
	{\subfigure[Initial spanning tree $T_0$ to start SA]{\includegraphics[width=0.45 \textwidth,height=0.45\textwidth]{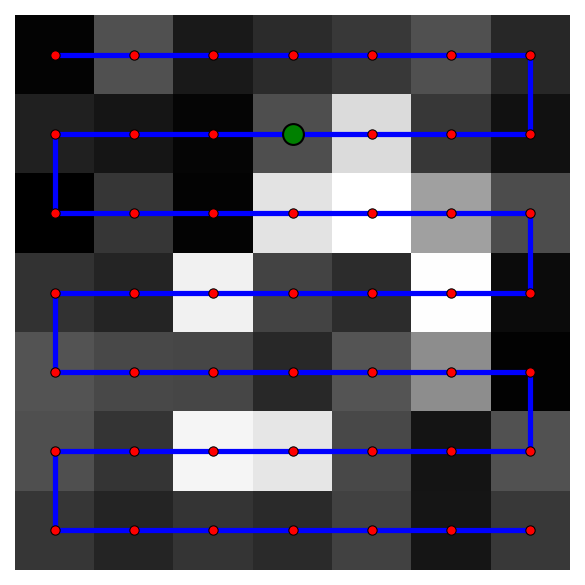}}}
	{\subfigure[Optimal spanning tree $T_\ast$   found by SA ]{\includegraphics[width=0.45 \textwidth,height=0.45\textwidth]{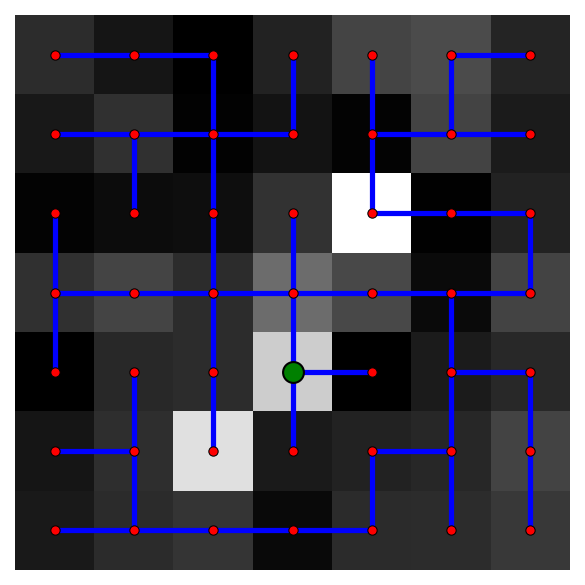}}}
	
	{\subfigure[Dual potential $u_{T_\ast}$ found using SA and formula \eqref{eq:pot_intro}]{\includegraphics[width=0.25\textwidth,height=0.25\textwidth]{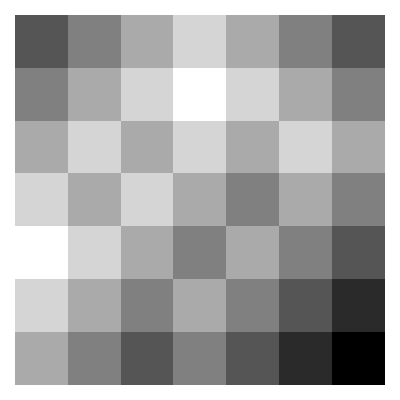}}}
	{\subfigure[Dual potential  $u_G$ found using POT and the ground cost $d_{G}$]{\includegraphics[width=0.25\textwidth,height=0.25\textwidth]{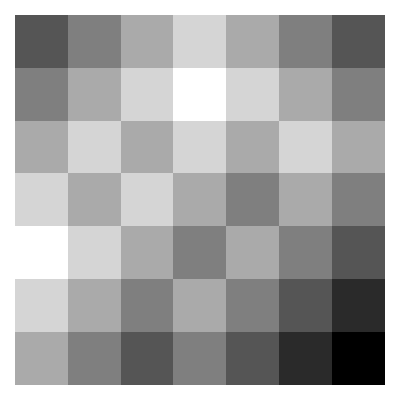}}}
	{\subfigure[Convergence of SA along $10^6$ iterations]{\includegraphics[width=0.4 \textwidth,height=0.25\textwidth]{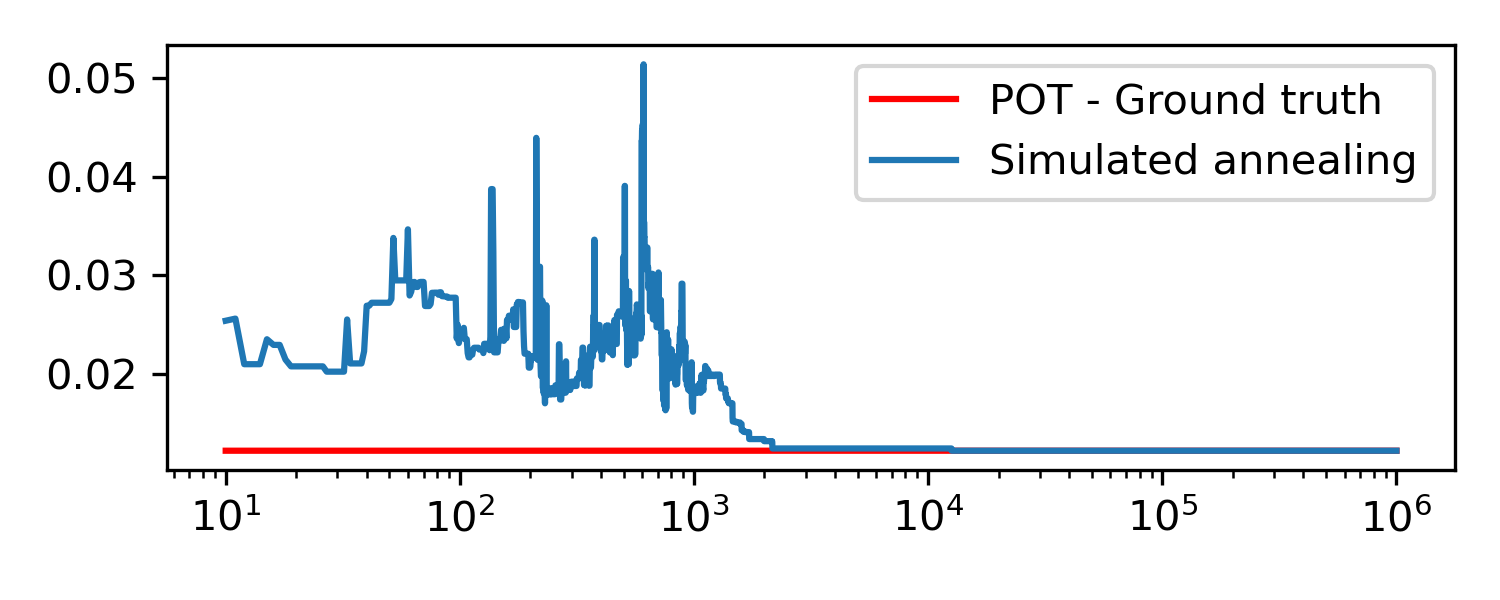}}}
	
	{\subfigure[Optimal coupling  $\gamma_{T_\ast}$ found using Algorithm  \ref{alg:tree_transport}]{\includegraphics[width=0.32 \textwidth,height=0.32\textwidth]{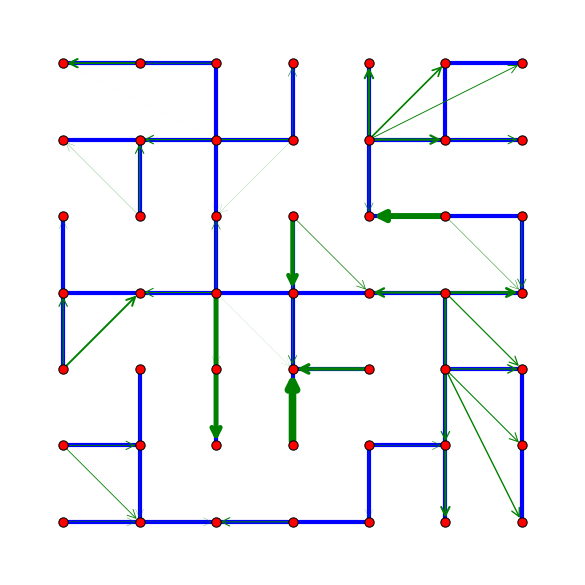}}}
	{\subfigure[Optimal coupling found using POT  and the ground cost $d_{T_\ast}$]{\includegraphics[width=0.32 \textwidth,height=0.32\textwidth]{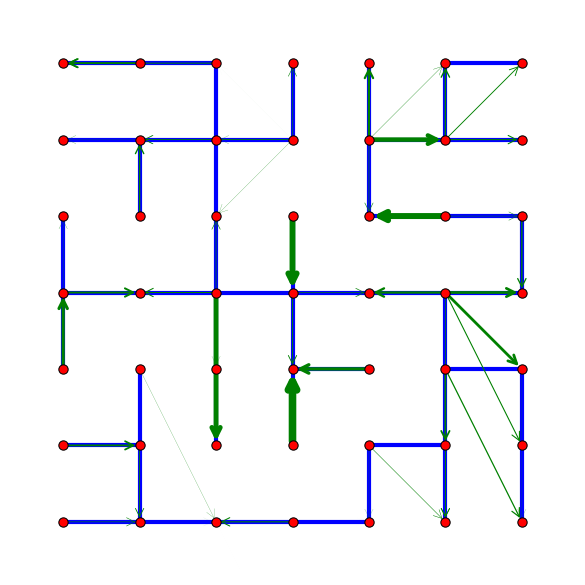}}}
	{\subfigure[Optimal coupling found using POT  and the ground cost $d_{G}$]{\includegraphics[width=0.32 \textwidth,height=0.32\textwidth]{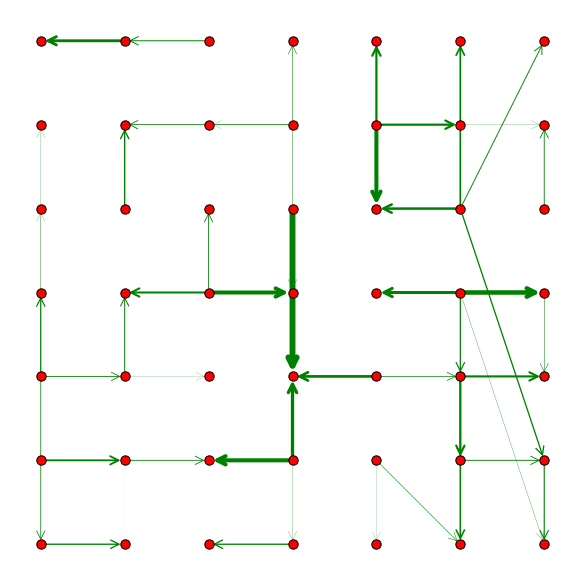}}}
	
	\caption{OT using Simulated Annealing (SA) in the noisy setting  $p=7$.
	}   \label{fig:mu0nu1_noise_small}
\end{figure}


\begin{figure}[htbp]
	\centering
	{\subfigure[Initial spanning tree $T_0$ to start SA]{\includegraphics[width=0.45 \textwidth,height=0.45\textwidth]{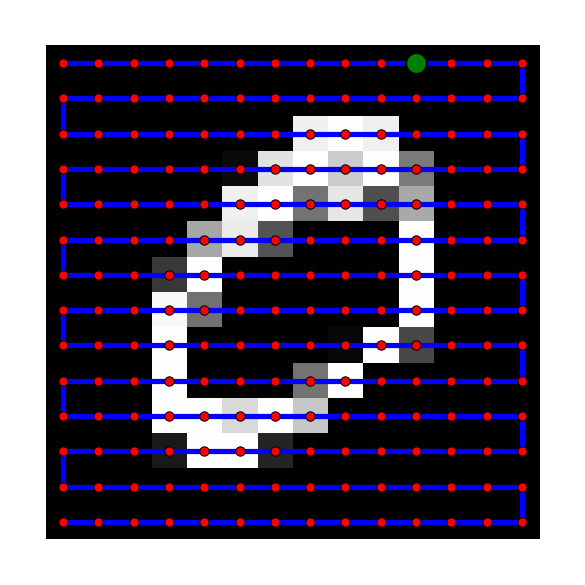}}}
	{\subfigure[Optimal spanning tree $T_\ast$   found by SA ]{\includegraphics[width=0.45 \textwidth,height=0.45\textwidth]{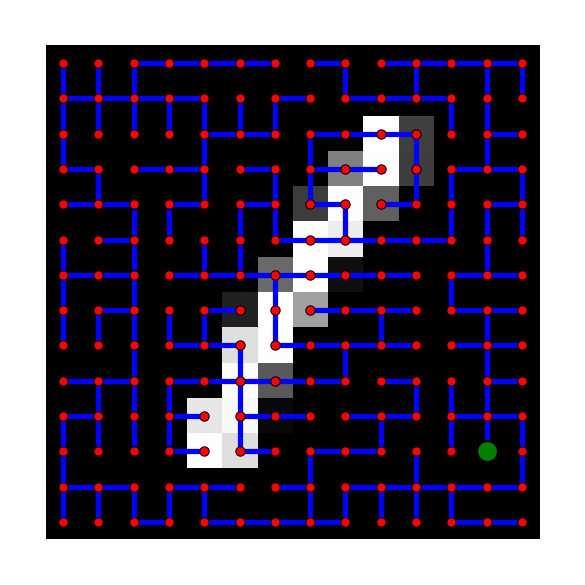}}}
	
	{\subfigure[Dual potential $u_{T_\ast}$ found using SA and formula \eqref{eq:pot_intro}]{\includegraphics[width=0.25\textwidth,height=0.25\textwidth]{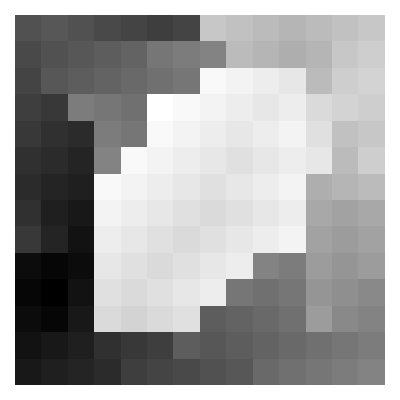}}}
	{\subfigure[Dual potential $u_G$ found using POT and the ground cost $d_{G}$]{\includegraphics[width=0.25\textwidth,height=0.25\textwidth]{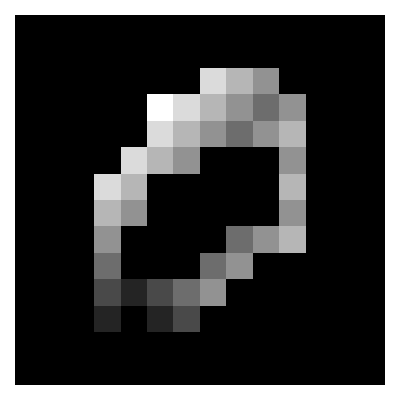}}}
	{\subfigure[Convergence of SA along $10^7$ iterations]{\includegraphics[width=0.4 \textwidth,height=0.25\textwidth]{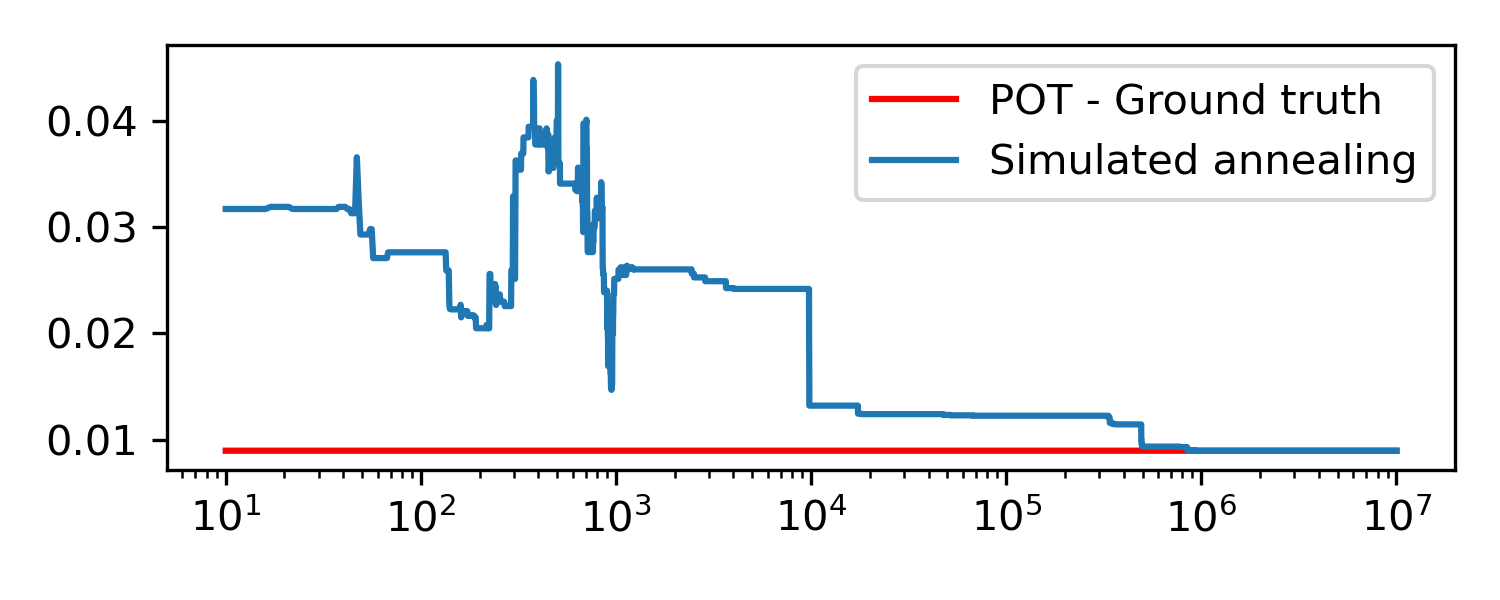}}}
	
	{\subfigure[Optimal coupling $\gamma_{T_\ast}$ found using Algorithm  \ref{alg:tree_transport} ]{\includegraphics[width=0.32 \textwidth,height=0.32\textwidth]{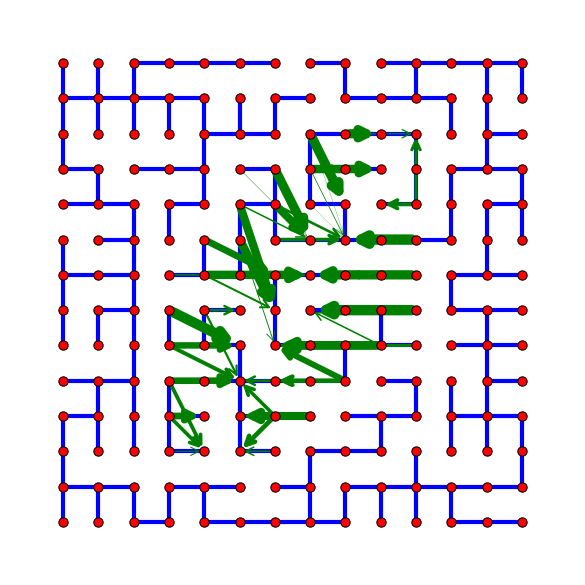}}}
	{\subfigure[Optimal coupling found using POT and the ground cost $d_{T_\ast}$]{\includegraphics[width=0.32 \textwidth,height=0.32\textwidth]{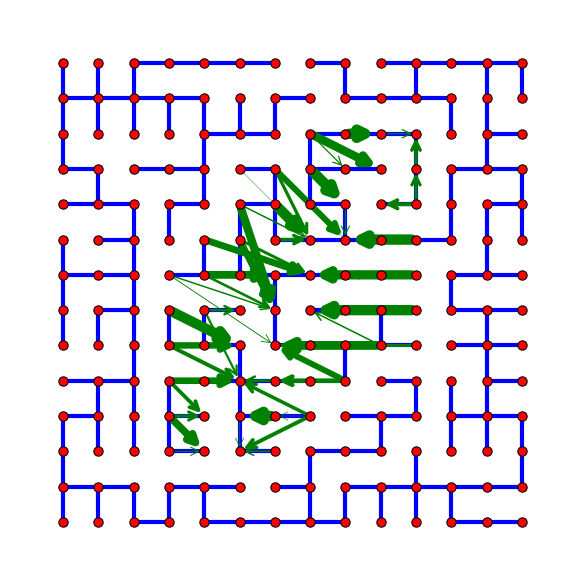}}}
	{\subfigure[Optimal coupling found using POT and the ground cost $d_{G}$]{\includegraphics[width=0.32 \textwidth,height=0.32\textwidth]{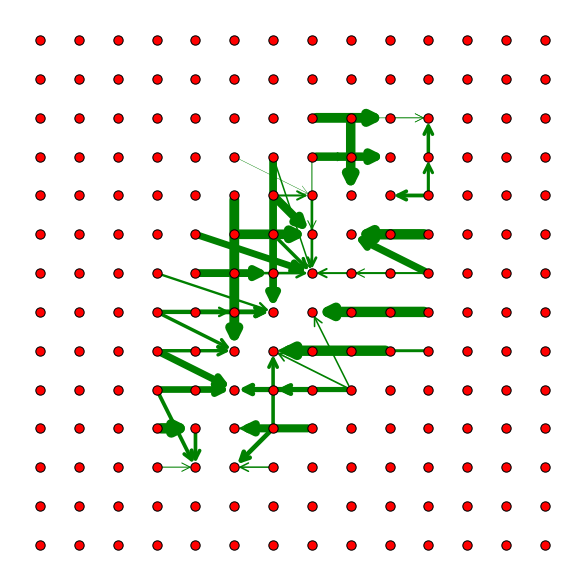}}}
	
	\caption{OT using Simulated Annealing (SA) in the noiseless setting with $p=14$.
	}   \label{fig:mu0nu1_noiseless_large}
\end{figure}

\begin{figure}[htbp]
	\centering
	{\subfigure[Initial spanning tree $T_0$ to start SA]{\includegraphics[width=0.45 \textwidth,height=0.45\textwidth]{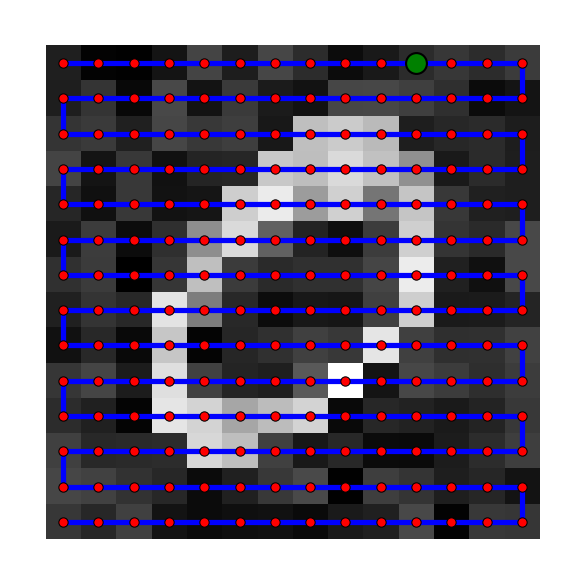}}}
	{\subfigure[Optimal spanning tree $T_\ast$   found by SA ]{\includegraphics[width=0.45 \textwidth,height=0.45\textwidth]{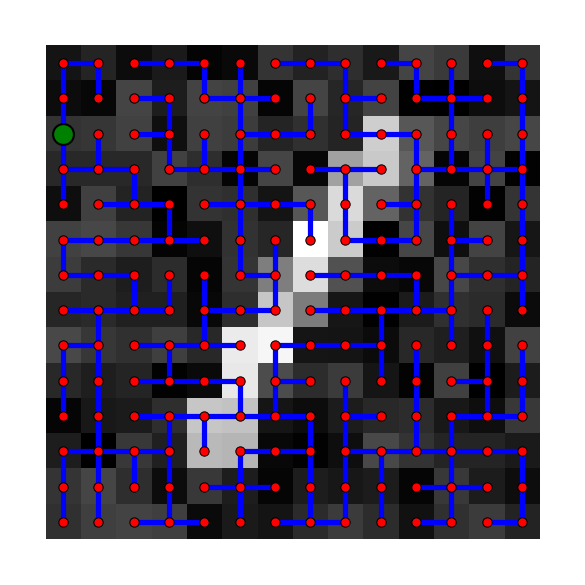}}}
	
	{\subfigure[Dual potential $u_{T_\ast}$ found using SA and formula \eqref{eq:pot_intro}]{\includegraphics[width=0.25\textwidth,height=0.25\textwidth]{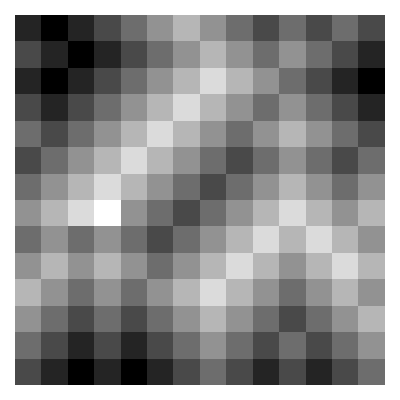}}}
	{\subfigure[Dual potential  $u_G$ found using POT and the ground cost $d_{G}$]{\includegraphics[width=0.25\textwidth,height=0.25\textwidth]{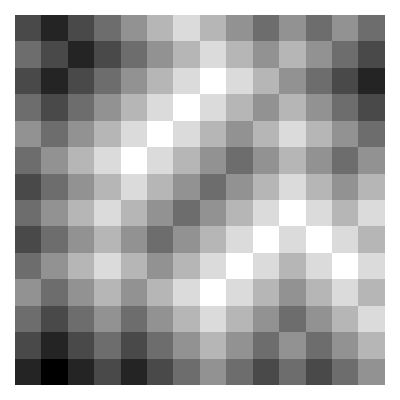}}}
	{\subfigure[Convergence of SA along $10^8$ iterations]{\includegraphics[width=0.4 \textwidth,height=0.25\textwidth]{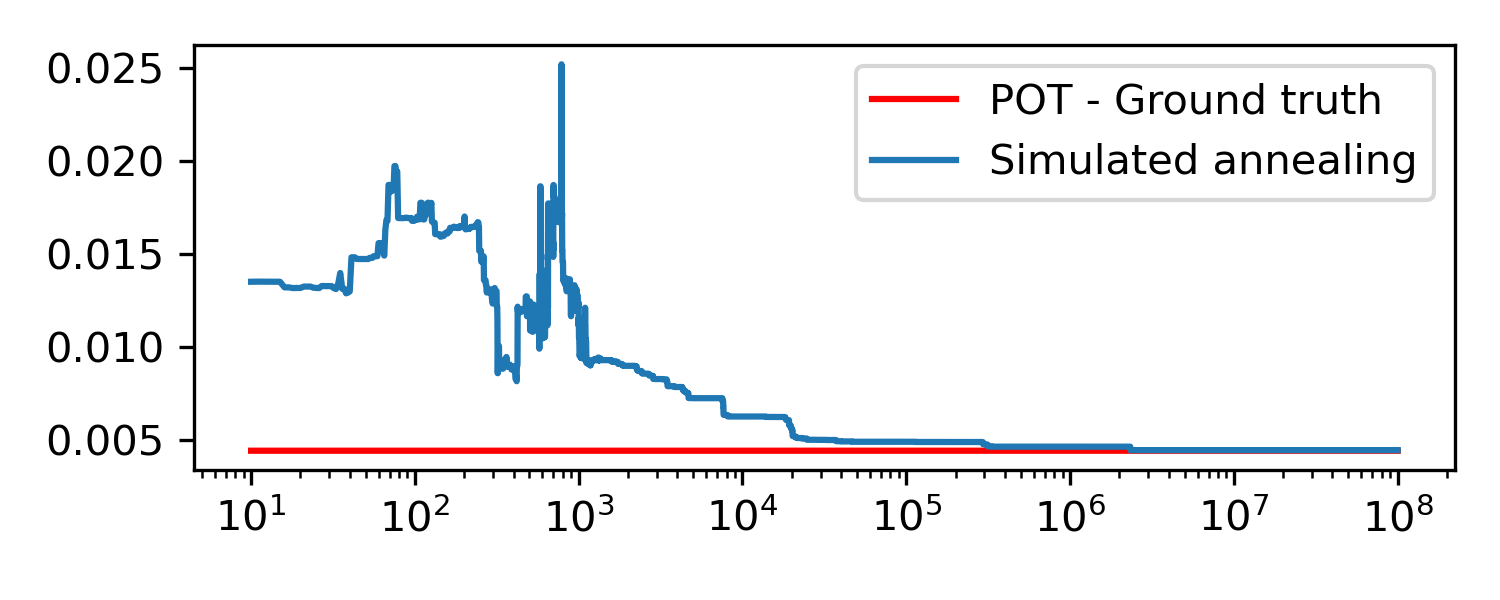}}}
	
	{\subfigure[Optimal coupling  $\gamma_{T_\ast}$ found using Algorithm  \ref{alg:tree_transport}]{\includegraphics[width=0.32 \textwidth,height=0.32\textwidth]{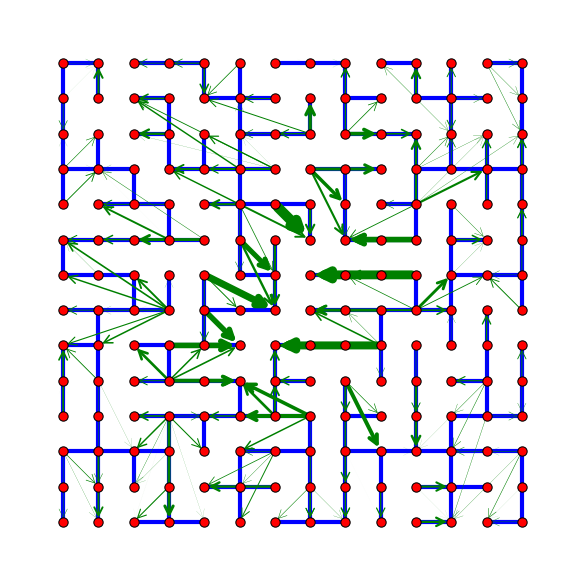}}}
	{\subfigure[Optimal coupling found using POT  and the ground cost $d_{T_\ast}$]{\includegraphics[width=0.32 \textwidth,height=0.32\textwidth]{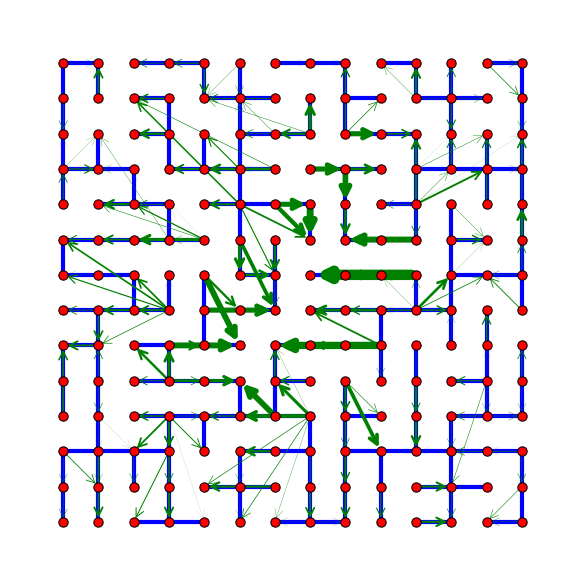}}}
	{\subfigure[Optimal coupling found using POT  and the ground cost $d_{G}$]{\includegraphics[width=0.32 \textwidth,height=0.32\textwidth]{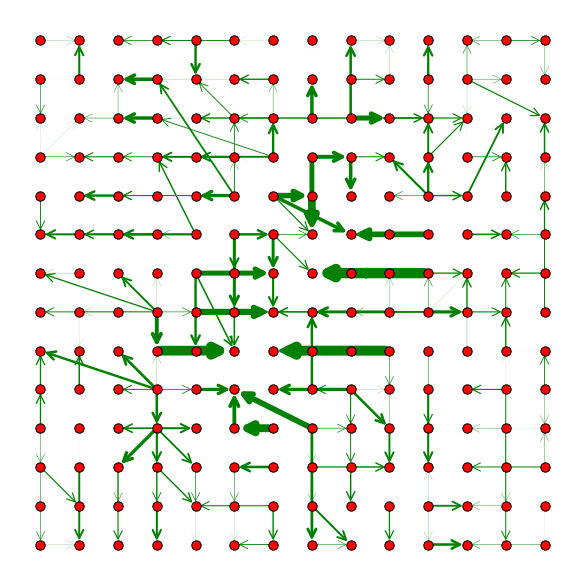}}}
	
	\caption{OT using Simulated Annealing (SA) in the noisy setting with $p=14$.}   
	\label{fig:mu0nu1_noise_large}
\end{figure}
\section{Conclusion and perspectives}\label{sec:conclusion}

We studied optimal transport between probability measures supported on the same finite metric space endowed with a graph-induced distance, leveraging a reformulation of the problem as a minimization over spanning trees. Under a weak non-degeneracy condition, we proved the uniqueness of the Kantorovich potential and derived an explicit expression based on an imbalanced cumulative function along an optimal spanning tree, revealing a close analogy with monotone transport on the real line. From a computational viewpoint, we proposed a stochastic algorithm based on simulated annealing to estimate the Kantorovich distance, together with a dynamic programming procedure to compute an optimal transport plan and a dual potential. Overall, our approach provides a computational alternative to classical network simplex and minimum-cost flow methods for optimal transport on graphs.

An interesting direction for future research concerns the theoretical analysis of the convergence of the simulated annealing algorithm for minimizing the OT functional over the space of spanning trees. Another interesting perspective is the application of the proposed framework to large-scale problems in computational optimal transport.

\bibliographystyle{siam}

\end{document}